\documentclass[10pt]{article}
\usepackage{amsmath,amssymb,amsthm, epsfig}
\usepackage{amsfonts}
\usepackage{amsmath}
\usepackage{amssymb}
\usepackage{color}
\usepackage{palatino}
\usepackage[mathscr]{eucal}

\title{Optimal design problems in rough inhomogeneous media.
Existence theory\footnote{This work is partially supported by
National Science Foundation Grant DMS 0600930.}}

\author{Eduardo V. Teixeira}

\date{Rutgers University, \\
 Department of Mathematics, \\
 Piscataway, NJ 08854-8019 \vspace{.5cm}\\
}

\newlength{\hchng}
\newlength{\vchng}
\def \R {\mathbb{R}}
\def \supp {\mathrm{supp } }
\def \div {\mathrm{div}}
\def \dist {\mathrm{dist}}
\def \redbdry {\partial_\mathrm{red}}

\def \suchthat {\ \big | \ }

\def \L {\mathfrak{L}}

\def \J {\mathfrak{J}}
\def \H {\mathcal{H}^{n-1}(X)}
\def \A {\mathcal{A}}
\def \Leb {\mathscr{L}^n}

\newtheorem{theorem}{Theorem}[section]
\newtheorem{lemma}[theorem]{Lemma}
\newtheorem{proposition}[theorem]{Proposition}

\theoremstyle{definition}

\theoremstyle{remark}

\numberwithin{equation}{section}

\newcommand{\intav}[1]{\mathchoice {\mathop{\vrule width 6pt height 3 pt depth  -2.5pt
\kern -8pt \intop}\nolimits_{\kern -6pt#1}} {\mathop{\vrule width
5pt height 3  pt depth -2.6pt \kern -6pt \intop}\nolimits_{#1}}
{\mathop{\vrule width 5pt height 3 pt depth -2.6pt \kern -6pt
\intop}\nolimits_{#1}} {\mathop{\vrule width 5pt height 3 pt depth
-2.6pt \kern -6pt \intop}\nolimits_{#1}}}

\begin{document}
\maketitle

\begin{abstract}
This paper settles the existence question for a rather general
class of convex optimal design problems with a volume constraint.
In low dimensions, we prove the existence of an optimal
configuration for general convex minimization problems ruled by
bounded measurable degenerate elliptic operators. Under a mild
continuity assumption on the medium, the free boundary is proven
to enjoy the appropriate weak geometry and we establish the
existence of an optimal design for general convex optimal design
problems with volume constraints for all dimensions.
\end{abstract}

\tableofcontents

\section{Introduction}

Well known for modeling important problems in applied mathematics,
respected for the challenging mathematical questions they give
rise to and admired for their intrinsic beauty, optimization
problems with volume constraints have received an overwhelming
attention in the past few decades. In general, the usual
techniques of the Calculus of Variations are not sufficiently
powerful, or even appropriate, to establish existence of optimal
configurations for those classes of problems. This fact has
inspired remarkable recent advances in a number of branches of
applied analysis in an attempt to develop the right set of
analytical and geometrical tools to study optimal design problems
with volume constraints.
\par
One of the fundamental motivations of this present work can be, in
its most basic form, stated as follows: given an $n$-dimensional
body and a fixed amount of insulating material, what is the best
way of insulating it? Depending on the flexibility allowed, the
mathematical set-up used to model this classic question can also
be employed in the analysis of a variety of other problems in
applied mathematics. In more precise mathematical terms, but still
using the language of heat conduction, the above question takes
the following form: let $D$ be a fixed Lipschitz bounded domain in
$\R^n$ (the body to be insulated), $\varphi \colon \partial D \to
\R$ be a prescribed positive function (the temperature
distribution on $D$), and $\iota > 0$ be a given positive number
(the amount of insulating material available). For each
configuration $\Omega$ that surrounds $D$ and obeys $\Leb (\Omega
\setminus D) \le \iota$, we compute the flux associated to it:
$$
    \Omega  \mapsto \J (\Omega).
$$
In general, $\J$ is related to a boundary integral involving a
potential $u_\Omega$, linked to $\Omega$ by a prescribed PDE. The
optimal design problem is then
\begin{equation}\label{General Optimization Prob. INTROC}
    \textrm{Min } \left \{ \J(\Omega) \suchthat \Omega \subset D \quad \textrm{ and
    } \quad
    \Leb (\Omega \setminus D) \le \iota \right \}.
\end{equation}
Probably the first and still one of the most influential works in
this line of research is the pioneering article of Aguilera, Alt
and Caffarelli, \cite{AAC}. In this paper, the authors address the
question of minimizing the Dirichlet integral when prescribed the
volume of the zero set. More precisely, they study the
optimization problem
\begin{equation}\label{Prob studied in AAC}
    \textrm{Min} \quad  \left \{ \int |\nabla u|^2dX
    \suchthat u \in H^1(\Omega),\quad u = \varphi \ge 0 \textrm{
    on } \partial \Omega \quad \textrm{ and } \quad \Leb \left ( \{ u = 0\}
    \right ) = \alpha \right \},
\end{equation}
for a fixed $\alpha < \Leb (\Omega)$. In the case of an exterior
domain, $\Omega = \R^n \setminus D$, problem (\ref{Prob studied in
AAC}) can be used to model a very simple, yet interesting optimal
design problem with volume constraint as stated above. Namely,
suppose $D$ is evenly heated. If one tries to minimize the heat
flux given by $\int_{\partial \Omega} u_\mu d\H$, where $u$ is the
capacity potential associated to $\Omega$,  with $\Leb \left
(\Omega \setminus D \right )$ prescribed, a simple application of
Green's identity reveals that the heat flux equals the Dirichlet
integral, and therefore the problem becomes identical to
(\ref{Prob studied in AAC}). Fine regularity properties of the
free boundary, $\partial \{u^\star > 0 \} \cap \Omega$, where
$u^\star$ is a minimizer of (\ref{Prob studied in AAC}) rely on
the powerful geometric-measure machinery developed by Alt and
Caffarelli in \cite{AC}: the \textit{magnum opus} of free boundary
regularity theory for variational problems.
\par
A significant generalization of  problem (\ref{Prob studied in
AAC}) was carried out by Lederman in \cite{Led}. In this paper,
the author studies the non-homogeneous minimization problem, that
is, the Dirichlet integral is replaced by $\int |\nabla u|^2dX -
\int g u$, for a given $g$ bounded away from zero.
\par
In an important paper, Ambrosio, Fonseca, Marcellini and Tartar,
\cite{AFMT}, address another major generalization of problem
(\ref{Prob studied in AAC}). Namely they establish the existence
of a minimizer to the functional ${\mathscr{F}}: = \int_\Omega
W(Du)dx$,  for $W \colon \R^{d\times n} \to (0,\infty)$ $C^1$ and
quasi-convex, with the multiple volume constraint
${\mathscr{L}}^n(\{u=z_i\})=\alpha_i, 0\leq i\leq k$. In a
subsequence article, Tilli, in \cite{Tilli}, showed, for $W(\xi)
:= |\xi|^2$, that in the case of just two level constraints, the
minimizers are locally Lipschitz continuous.
\par
Still assuming a constant temperature distribution, Oliveira and
the author in \cite{OT} studied the optimization problem
(\ref{General Optimization Prob. INTROC}), governed by the
$p$-Laplacian operator when the flux is given by $\J(u) :=
\int_{\partial \Omega} \left (u_\mu \right )^{p-1} d\H$. This
translates into the analysis of the minimization problem
(\ref{Prob studied in AAC}), for the $p$-Dirichlet integral, that
is, $W(\xi) = |\xi|^p$, for $p>1$.
\par
The first work to deal with optimal design problems with
non-constant temperature distribution $\varphi \colon \partial D
\to (0,\infty)$ is \cite{ACS}. In this paper, the authors consider
the linear functional: $\J(\Omega) = \int \Delta u dX$, where $u$
is the harmonic function in $\Omega \setminus D$, taking boundary
data $\varphi$ on $\partial D$ and zero on $\partial \Omega$. Even
for this simple functional, major difficulties arise. For
instance, the free boundary condition, that is, the behavior of
$\nabla u^\star$ along the free boundary, $\partial \Omega^\star$,
is non-local and it required a new machinery to establish the
appropriate geometric-measure properties of the free boundary
necessary to perform suitable smooth perturbations. The latter is
used in its entirely to finally conclude the existence of an
optimal design.
\par
At least for smooth competing configurations, $\Omega$, for the
linear functional studied in \cite{ACS} we have
$$
    \J(\Omega) :=  \int \Delta u dX = \int_{\partial \Omega} u_\nu
    d\H = \int_{\partial D} u_\mu d\H.
$$
This is a na\"ive, yet important observation, as the latter
integral is taken over the fixed boundary. Therefore, at least in
an intuitive perspective, a non-linear theory for this class of
minimization problems should use $\int_{\partial D} u_\mu d\H$ as
its linear pattern. From the applied viewpoint, if one allows a
nonlinear flux, $\J$ that might also depend upon the local
structure of the boundary of the body $D$, i.e.,
\begin{equation}\label{General Nonlinear Functional - INTRO}
    \J(\Omega) := \int_{\partial D} \Gamma \big (X, u_\mu(X) \big ) d\H
\end{equation}
the mathematical model (\ref{General Optimization Prob. INTROC})
would address several other physical situations, such as: optimal
configurations in electrostatics, problems in material science,
flux dynamics, among many others. This nonlinear setting, however
still only for problems governed by the Laplacian operator, has
been studied by the author in \cite{Teix} and \cite{T3}.
\par
In this present paper, we settle the existence theory for optimal
design problem (\ref{General Optimization Prob. INTROC}) with
nonlinear functionals as in (\ref{General Nonlinear Functional -
INTRO}), when $u_\Omega$ is linked with $\Omega$ by a rather
general class of degenerate elliptic PDEs.  In terms of
applications, it greatly extends the range of physical systems
that can be modeled by this set-up. From the mathematical
viewpoint, this project brings a number of new rather challenging
difficulties in its analysis and modern solutions to various
issues commonly found in free boundary problems are developed
throughout the paper. Free boundary regularity theory for uniform
elliptic operators in divergence form with merely H\"older
continuous coefficients is currently being developed in order to
establish $C^{1,\gamma}$ smoothness of an optimal configuration,
up to a possible negligible singular set, \cite{Teix-Prep}.
\par
The article is organized as follows: in section \ref{SECTION -
Mathematical Set-up}, we describe all the mathematical elements
involved in the model and the optimization problem is accurately
stated in that section. Still in section \ref{SECTION -
Mathematical Set-up}, we introduce weak formulations of the
optimal design problem (\ref{General Optimization Prob. INTROC})
that are somewhat simpler to be tackled from the mathematical
perspective. Basic properties of the functional to be minimized
are established in section \ref{SECTION Basic functional and
analytic properties}. The first existence theorem for a weak
formulation of the original optimization problem is delivered in
section \ref{SECTION - Existence of minimizer to problem
Penalized}. In section \ref{SECTION - Existence up to dimension
p}, by letting the penalty term blow-up, we establish the
existence of an optimal configuration to the optimal design
problem with volume constraint (\ref{General Optimization Prob.
INTROC}) ruled by totally discontinuous degenerate elliptic
operators. For that though, a technical restriction on the
dimension is necessary. In section \ref{SECTION - Weak Geometric
Properties of the free boundary}, under $C^\epsilon$ regularity on
the medium, a series of results concerning the weak geometric
properties of the boundary of an optimal configuration to the weak
formulation of the original problem (\ref{General Optimization
Prob. INTROC}) are achieved. These are used in section
\ref{SECTION - Existence theory in any dimension} to ultimately
derive existence of an optimal configuration in all dimensions.

\section{Mathematical set-up} \label{SECTION - Mathematical
Set-up}

Throughout the paper, $D$ denotes a fixed Lipschitz bounded domain
in $\R^n$, $\varphi \colon \partial D \to \R$ is a prescribed
positive function and $\iota
> 0$ is a given positive number. Our medium deformation will be expressed by $\A \colon
D^C \times \R^n \to \R^n$, a measurable $p$-degenerate elliptic
map, that is,
\begin{description}
    \item[(a)] for each $\xi \in \R^n$, the mapping $X \mapsto
    \A(X, \xi)$ is measurable.
    \item[(b)] For a.e. $X \in D^C$, the mapping $\xi \mapsto  \A(X, \xi)$ is
    continuous.
    \item[(c)] There exists constants $0<\lambda \le \Lambda < \infty$ and a $p>1$, such
    that, for a.e. $X\in D^C$ and all $\xi \in \R^n$,
        \begin{description}
            \item[(i)] $\A(X,\xi) \cdot \xi \ge \lambda |\xi|^p$,
            \item[(ii)] $|\A(X,\xi)| \le \Lambda |\xi|^{p-1}$,
            \item[(iii)] $\langle \A(X,\xi_1) - \A(X,\xi_2), \xi_1
            - \xi_2 \rangle > 0$, whenever $\xi_1 \neq \xi_2$ and
            \item[(iv)] $\A(X, \alpha \xi) = \alpha |\alpha|^{p-2}
            \A(X,\xi)$.
        \end{description}
\end{description}
A typical example to keep in mind is
$$
    \A(X,\xi) = A(X) |\xi|^{p-2} \xi,
$$
with $\A$ bounded measurable, which gives rise to the theory of
optimal shape problems governed by the $p$-Laplacian in a totally
discontinuous medium.
\par
Our optimization problem is then formulated as follows: for each
domain $\Omega \subset D$ satisfying
\begin{equation}\label{volume constraint}
    \Leb \left ( \Omega \setminus D \right ) \le \iota,
\end{equation}
we consider the $\A$-potential, $u = u(\Omega)$, with the
prescribed boundary value $\varphi$ on the fixed boundary
$\partial D$, associated to $\Omega$, i.e. the unique solution to
\begin{equation}\label{A potential}
\left \{
\begin{array}{rll}
\L u := \div \big ( \A(X, Du) \big ) &=& 0 \textrm{ in } \Omega \setminus D \\
u & = & \varphi \textrm{ on } \partial D \\
u &=& 0 \textrm{ on } \partial\Omega
\end{array}
\right.
\end{equation}
and compute
$$
    \J(\Omega) := \displaystyle \int_{\partial D} \Gamma \left ( X, \partial_\A u(X) \right )
    d\H \quad \textrm{ (the flux: quantity to be
    minimized)}.
$$
Here $\Gamma \colon \partial D \times \R \to \R$ is a given
function, whose properties will be described soon, and
\begin{equation}\label{A-derivative}
    \partial_\A u(X) := \left \langle \A(X, \nabla u(X)) , \mu(X) \right \rangle
\end{equation}
where $\mu$ denotes the inward normal vector defined
$\mathcal{H}^{n-1}$ a.e. on $\partial D$. The optimal design
problem we are interested in is the following:
\begin{equation}\label{Minimization Problem}
    \textrm{Minimize } \quad \Big \{ \J(\Omega) \suchthat \Omega \supset D \textrm{ and
    }\Leb \left ( \Omega \setminus D \right ) \le \iota \Big \}.
\end{equation}
\par
The analytical (and naturally mild) properties assumed on the
nonlinearity $\Gamma$ are:
\begin{enumerate}
\item For each $X \in \partial D$ fixed, $\Gamma(X,\cdot)$ is
convex and increasing.

\item For each $t \in \mathbb{R}$ fixed, $\partial_t \Gamma(\cdot,t)$ is
continuous.

\item If $\Gamma(X_0,t_0) = 0$ then $\Gamma(Y,t_0) = 0~\forall Y \in \partial D$; otherwise
$
    \dfrac{\Gamma(Y,t)}{\Gamma(X,t)} \le L,
$ for a universal constant $L>0$.
\end{enumerate}
\par
Notice that from 1 the following coercivity condition holds:
\begin{equation}\label{coersivity}
\lim\limits_{t\to +\infty} \int_{\partial D} \Gamma(X,t)d\H =
+\infty.
\end{equation}
\par
If $\psi$ is a positive continuous function defined on $\partial
D$ and $\gamma$ is a increasing convex function, then
$$
    \Gamma(X,t) = \psi(X)\gamma(t)
$$
gives a typical nonlinearity that fulfils the above properties. As
in the Calculus of Variations, $\Gamma$ is chosen based upon the
particular problem we are trying to model and no relation
whatsoever is imposed upon the nonlinearity $\Gamma$ and $\A$.
\par
Sometimes it is convenient to use the language of heat conduction
theory to describe the elements involved in our analysis. Thus,
$D$ is the body to be insulated, $\varphi$ represents the
temperature distribution on $\partial D$, $\iota$ corresponds to
the maximum amount of insulating material available, $\J$ plays
the role of the (generalized) heat flux, which is the quantity to
be minimized, and $\A$ determines the inhomogeneous and complexity
features of the medium. However it is important to highlight that
this model is widely applicable to several other situations beyond
the bounds of the classical heat conduction theory and other
interpretation of the model might provide different insights on
what is reasonable to expect to hold.
\par
It is noteworthy to point out that, since we are not forcing any
regularity assumption on the medium $\A$, in principle just
H\"older continuity estimates are available for an $\A$-potential
$u = u(\Omega)$. Thus, the $\A$-normal derivative of $u$,
$\partial_\A u$, as entitled in (\ref{A-derivative}) is not
properly defined. Some of our primary results concerning geometric
properties of the free boundary will not depend upon any
smoothness condition on the medium. However, just to grapple with
this technical inconsistence, we will assume throughout the paper
that there exists a small $1>\!\!> \delta_0 > 0$, such that
\begin{equation} \label{Continuity of A close to D}
    \begin{array}{l}
        \textrm{(i) } \A \textrm{ is H\"older (or even only Dini) continuous  in }
        D_{\delta_0} := \left \{ X \in \R^n \suchthat \textrm{dist}(X,~\partial D) < \delta_0
        \right \}
        \textrm{ and } \\
        \textrm{(ii) } \varphi\colon \partial D \to \mathbb{R} \textrm{ is accordiantly smooth}.
    \end{array}
\end{equation}
Once more we emphasize that for the first part of this project,
condition (\ref{Continuity of A close to D}) plays merely a
technical role and, for sake of applications, it should not be
seen as a constraint.
\subsection{Penalty Method and weak formulation}

From the mathematical point of view, the minimization problem
(\ref{Minimization Problem}) carries too many difficulties to be
approached directly. Instead, we will employ a fruitful penalty
method in order to formulate weak versions of problem
(\ref{Minimization Problem}). Such a technique has been
successfully employed to study a variety problems in applied
mathematics.
\par
The intuitive idea behind a penalization strategy is the
following: suppose our problem has an ``undesired" (from the
mathematical perspective) constraint on the competing
configurations (in our case a volume constraint). We then allow
any configuration to compete; however we ``charge a fee" for those
configuration that do not obey the previously set constraint. We
expect that, if the fee is too high, optimal configurations will
indeed prefer to satisfy the original constraint.
\par
Still in a philosophical perspective, one should expect that an
optimal configuration, $\Omega^{\star}$, of problem
(\ref{Minimization Problem}) satisfies
$$
   \Leb \left ( \Omega^{\star} \setminus D \right ) = \iota.
$$
For that, think of $\iota$ as the budget available and
$\Omega^{\star}$ as the ultimate object to be built up.
Mathematically, this fact is indeed easily justified. For instance
suppose, for an optimal configuration $\Omega^{\star}$, we had
$$
    \Leb \left ( \Omega^{\star} \setminus D \right ) < \iota - \varepsilon,
$$
for some $\varepsilon > 0$. Let $X_0 \in \partial \Omega^{\star}$
be a free boundary point and $\rho > 0$ so that $\omega_n \rho^n <
\varepsilon$. Consider
$$
    \tilde{\Omega} := \Omega^{\star} \cup B_\rho(X_0).
$$
Thus, $\tilde{\Omega}$ competes with $\Omega^\star$ in the
minimization problem (\ref{Minimization Problem}) and, because of
maximum principle, $u  ( \tilde{\Omega} ) >  u ( {\Omega^\star})$.
Taking into account that $\Gamma$ is increasing and applying Hopf
maximum principle on $\partial D$, we would conclude
$$
    \J(\Omega^\star) > \J(\tilde{\Omega}),
$$
which contradicts the minimality property of $\Omega^\star$.  Our
conclusion is that in problem (\ref{Minimization Problem}) we can
regard the condition $\Leb ( \Omega \setminus D)  \le \iota$ as
$\Leb (\Omega \setminus D) = \iota$. For future reference, let us
state this as a Lemma.
\begin{lemma}\label{Lemma 1 - less or equal replaced by equal} Let
$\Omega^\star$ be a minimizer of problem (\ref{Minimization
Problem}). Then $\Leb (\Omega^\star \setminus D ) = \iota$.
\end{lemma}
Another general comment: we will always extend the $\A$-potential
$u(\Omega)$ by zero outside $\Omega$. Thus, in the distributional
sense,
\begin{equation}\label{superharmoniticity of A-pontentials}
    \L \left [ u(\Omega) \right ] = 0, \textrm{ in } \Omega = \{ u(\Omega) > 0 \} \quad
    \textrm{ and } \quad \L \left [ u(\Omega) \right ] \ge 0, \textrm{ in } \R^n
    \setminus D.
\end{equation}
\par
Returning to the penalty technique issue: we shall borrow the
simple, yet quite clever penalty term suggested in \cite{Tilli},
that is, for each $\lambda > 0$, we will consider the penalization
term $\varrho_\lambda \colon \R_+ \to \R_+$, defined by
\begin{equation}\label{penalization term}
\varrho_\lambda(t) := \lambda (t - \iota)^{+}.
\end{equation}
We then define the $\lambda$-perturbed functional, $\J_\lambda$,
to be
\begin{equation} \label{e-Functional}
    \J_\lambda(\Omega) := \displaystyle \int_{\partial D}
    \Gamma \left ( X, \partial_\A u(X) \right ) d\H + \varrho_\lambda
    \left ( \Leb \left ( \Omega \setminus D \right ) \right ).
\end{equation}
Once more, the idea is the following: we allow $\J_\lambda$ to act
on any configuration $\Omega \supset D$ and, when $\lambda$ is big
enough, we hope that an optimal design $\Omega^\star_\lambda$ for
$\J_\lambda$ will satisfy $\Leb \left ( \Omega^\star_\lambda
\setminus D \right ) = \iota$, thus it will also be a minimizer
for our original optimization problem with volume constraint. Our
initial goal is then study existence and geometric properties of
the penalized problem:
\begin{equation}\label{Penalized Min Problem}
    (\mathfrak{P}_\lambda) \quad \quad
    \textrm{Minimize } ~ \Big \{ \J_\lambda(\Omega) \textrm{ among all sets }
    \Omega \supset D \Big \}.
\end{equation}
However, even the penalty problem (\ref{Penalized Min Problem})
is, in principle, too hard to be directly approached. Thus, for
the time being, it will be more appropriate to initially deal with
a weak formulation of problem (\ref{Penalized Min Problem}), which
we start describing now. Let $\delta_0$ be the technical number in
(\ref{Continuity of A close to D}). For each $\delta <\!\!<
\delta_0$, we us define the functional set
\begin{equation}\label{def of V delta funct set}
    \mathcal{V}(\delta) := \left \{ f \in W^{1,p}(D^C) \suchthat f =
    \varphi \textrm{ on } \partial D, ~ f \ge 0, ~ \mathfrak{L} f \ge 0, ~ \mathfrak{L} f =
    0 \textrm{ in } D_\delta \right \}.
\end{equation}
Then we define the sample functional set:
\begin{equation}\label{def of V funct set}
    \mathcal{V} := \bigcup\limits_{\delta \searrow 0} \mathcal{V}(\delta)
\end{equation}
and the weak formulation of problem (\ref{Penalized Min Problem})
can then be stated as
\begin{equation} \label{Weak Min Problem}
    (\mathfrak{P}_\lambda^{\mathrm{weak}}) \quad \quad
    \min\limits_{f \in \mathcal{V}}  \left \{
    \int_{\partial D} \Gamma \left ( X, \partial_\A f(X) \right )
    d\H + \varrho_\lambda \left ( \Leb \left (\{ f > 0 \} \setminus D \right ) \right )
    \right \}.
\end{equation}
\section{Basic functional and analytic properties}
\label{SECTION Basic functional and analytic properties}

In this section we establish all the basic and necessary
properties on the mathematical elements of the problems we are
interested in, namely, problems (\ref{Minimization Problem}),
(\ref{Penalized Min Problem}) and (\ref{Weak Min Problem});
however we will mostly be concerned with the latter, as it the the
weakest formulation among them all.
\par
We start by stating, as a lemma, a simple yet crucial observation
regarding the measure theory involved on our optimization
problems. The proof is somewhat long, but rather standard and we
omit it here.
\begin{lemma} \label{Lf is a Radon Measure} Let $f \in
\mathcal{V}$, as in (\ref{def of V funct set}). Then $\mathfrak{L}
f$ defines a nonnegative Radon measure, $\mu_f$, in $D^C$. In
particular, for any $\psi \in C(D^C) \cap W^{1,p}(D^C)$,
$$
    \int_{B_r(Y)} \psi(X) d\mu_f(X) + \int_{B_r(Y)} \langle \A(X,Df), D \psi(X)
    \rangle dX = \int_{\partial B_r(Y)} \psi(S) \cdot
    \partial_\A f(S)d\mathcal{H}^{n-1}(S),
$$
for almost all $0 \le r < \dist (Y,\partial D)$. Also, if $\psi
\in W^{1,p}_0(D^C)$, there holds
\begin{equation}\label{weak Div Thm}
    \int_{D^C} \langle \A(X,Df), D\psi(X) \rangle dX = \int_{D^C} \psi(X) d\mu_f(X).
\end{equation}
Furthermore,
\begin{equation}\label{meaure mu is normal bdry int}
    \mu_f \left (\mathbb{R}^n \setminus D \right ) = \int_{\partial
    D} \partial_\A f(X) d\H.
\end{equation}
\end{lemma}
Another useful results that we make use throughout the paper is
(for a proof in the case of the $p$-Laplacian we refer, for
instance, to \cite{DP}, page 100):
\begin{lemma} \label{p harmonic functions with same bondary data}
Let $\mathcal{O}$ be a domain in $\R^n$ and $f \in
W^{1,p}(\mathcal{O})$. There exists a constant $c = c(n,\A) > 0$,
such that
$$
    \int_{\mathcal{O}} \left ( \langle \A(X, Df), Df \rangle  - \langle \A(X, Dh), Dh \rangle
    \right ) dX \ge c
    \left \{
        \begin{array}{lll}
            \displaystyle \int_{\mathcal{O}} |\nabla (f-h)|^p dX &\textrm{ if
            }& p\ge 2 \\
            \alpha(f) \cdot \left [
            \displaystyle \int_{\mathcal{O}} |\nabla (f-h)|^p dX \right
            ]^{2/p} &\textrm{ if }& 1 < p \le 2.
        \end{array}
    \right.
$$
where
$$
    \alpha (f) := \displaystyle \left [ \int_{\mathcal{O}} |\nabla f|^p dX\right]^{1-\frac{2}{p}}
$$
and $h$ is the $\A$-harmonic function in $\mathcal{O}$ that
agrees with $f$ on $\partial \mathcal{O}$.
\end{lemma}
\par
Our first Proposition provides an energy estimate for a minimizing
sequences to our optimization problems. More precisely, we have:
\begin{proposition}\label{Energ Estimate} Let $u_j$ be a
minimizing sequence for the functional $\J_\lambda$. Then,
$$
    \|\nabla u_j \|_{L^p(D^C)} \le C,
$$
where $C$ depends only on dimension, $\A$, $D$, $\varphi$ and
$\Gamma$.
\end{proposition}
\begin{proof}
Let $\mathfrak{h} = \mathfrak{h}_p$ be the $p$-harmonic function
in $D^C$ that agrees with $\varphi$ on $\partial D^C$, that is the
solution to
\begin{equation}\label{A potential on D complementar}
    \left \{
        \begin{array}{rll}
            \Delta_p \mathfrak{h} &=& 0 \textrm{ in } D^C \\
            \mathfrak{h} &=& \varphi \textrm{ on } \partial D^C \\
            \mathfrak{h} &\in& W^{1,p}(D^C).
        \end{array}
    \right.
\end{equation}
From the maximum principle, there holds
$$
    0 \le  \mathfrak{h} \le \sup\limits_{\partial D} \varphi.
$$
For sake of notation convenience, let us denote $\L u_j dX :=
\mu_{u_j} = \mu_j $, as in Lemma \ref{Lf is a Radon Measure}. We
clearly have
\begin{equation} \label{Energ Est Eq01}
    \begin{array}{lll}
        \displaystyle\int (\mathfrak{h} - u_j) d\mu_j &=&
        \displaystyle\int \langle \A(X, Du_j), (\mathfrak{h} -
        u_j) \rangle dX \\
        &=& \displaystyle\int \langle \A(X,Du),Dh(X) \rangle dX  - \int \langle
        \A(X,Du), Du \rangle dX.
    \end{array}
\end{equation}
From the degenerate ellipticity of $\A$, we can deduce from
(\ref{Energ Est Eq01}) that
\begin{equation}\label{Energ Est Eq02}
    \begin{array}{lll}
        \lambda \displaystyle \int |D u_j(X) |^p dX & \le &
        \left | \displaystyle \int (\mathfrak{h} - u) d\mu_j -
        \displaystyle \int \langle \A(X,Du), D \mathfrak{h} \rangle dX \right |
        \\
        & \le & \sup\limits_{\partial D} \varphi \cdot \mu_j(\R^n
        \setminus D) + \Lambda \displaystyle \int |Du_j|^{p-1} |D\mathfrak{h}|
        dX \\
        & \le & \sup\limits_{\partial D} \varphi \cdot \mu_j(\R^n
        \setminus D) + \dfrac{\lambda}{2}  \displaystyle \int |Du_j|^p dX +
        C_1  \displaystyle \int |D\mathfrak{h}|^p dX.
    \end{array}
\end{equation}
In the last step we have used Young's inequality and $C_1 =
\epsilon^{-p}/p$ where $\epsilon$ satisfies $\epsilon^{p/(p-1)} =
p\lambda /2(p-1)$. In view of (\ref{meaure mu is normal bdry int})
and the estimate in (\ref{Energ Est Eq02}), we reach the
conclusion that there exists a constant $C_1$, depending only on
$\A$, $D$ and $\varphi$, such that
\begin{equation}\label{Energ Est Eq03}
    \|\nabla u_j \|_{L^p(D^C)}^p \le C_1 \left ( 1 + \dfrac{1}{2\alpha} \int_{\partial
    D} \partial_\A u_j(X) d\H \right ),
\end{equation}
where $\alpha := \mathcal{H}^{n-1}(\partial D)$. From the
monotonicity and convexity properties of the non-linearity
$\Gamma$, we derive, for each $Y \in \partial D$ fixed, that
$$
    2 \Gamma \left( Y, \|\nabla u_j \|_{L^p(D^C)}^p \right ) \le
    C_2 + \Gamma \left( Y, \dfrac{1}{\alpha} \int_{\partial
    D} \partial_\A u_j(X) d\H \right ),
$$
where $C_2$ is a constant depending only on $\A$, $D$, $\varphi$
and $\Gamma$. Once more using the convexity of $\Gamma(Y, \cdot)$,
it follows from Jensen's inequality that
\begin{equation}\label{Energ Estimate - Ineq 3}
    2 \Gamma \left( Y, \|\nabla u_j \|_{L^p(D^C)}^p \right ) \le
    C_1 +  \dfrac{1}{\alpha} \int_{\partial D}\Gamma \left( Y,\partial_\A u_j(X) \right )
    d\H.
\end{equation}
Integrate inequality (\ref{Energ Estimate - Ineq 3}) with respect
to $Y$ over $\partial D$ and taking into account property (iii) of
the non-linearity $\Gamma$, we derive
\begin{equation}\label{Energ Estimate - Ineq 4}
    \begin{array}{lll}
        \displaystyle \int_{\partial D} \Gamma \left( Y, \|\nabla u_j \|_{L^p(D^C)}^p \right )
        d\mathcal{H}^{n-1}(Y) &\le& C_3 \left ( 1 +
        \displaystyle  \int_{\partial D}\Gamma \left( X,\partial_\A u_j(X) \right )
        d\H \right ),
    \end{array}
\end{equation}
where again $C_3$ depends only upon $\A$, $D$, $\varphi$ and
$\Gamma$. Finally, (\ref{Energ Estimate - Ineq 4}) and the
coercivity of the function
$$
    t \mapsto \int_{\partial D}\Gamma \left( X, t \right )
    d\H,
$$
see (\ref{coersivity}), together complete the proof of the
Proposition.
\end{proof}
In view of the energy estimate provided in Proposition \ref{Energ
Estimate}, it becomes natural to investigate the behavior of
$\J_\lambda$ over weakly convergent sequences in $W^{1,p}$. In
this direction we have
\begin{lemma}\label{semi-continuity}
Let $f_j \in H^1(D^C)$ be a sequence of functions satisfying, $\L
f_j \ge 0$ in the distributional sense and, for some $\delta > 0$,
$\L f_j = 0$ in $D_\delta := \left \{X \in D^C \suchthat \dist (X,
\partial D) < \delta \right \}.$ Assume $f_j$ converges weakly to
$f$ in $W^{1,p}(D^C)$. Then, $\L f \ge 0$ in the distributional
sense, $\L f = 0$ in $D_\delta$ and furthermore
$$
    \J (f) + \varrho_\lambda \left ( | \{ f > 0 \} | \right )
    \le \liminf\limits_{j \to \infty}  \Big \{
    \J(f_j) + \varrho_\lambda \left ( | \{ f_j > 0 \} | \right )  \Big \}.
$$
\end{lemma}
\begin{proof}
The fact that $\L f \ge 0$ in the distributional sense follows
easily. In fact, for any nonnegative $\psi \in C^1_0(D^C)$, we
have
$$
    \langle \L f, \psi \rangle := - \int \langle \A(X,Df), D \psi
    \rangle dX = - \lim\limits_{j\to \infty} \int \langle \A(X,Df_j), D \psi
    \rangle dX \ge 0,
$$
since, for any $j$,  $\lim\limits_{j\to \infty} \int \langle
ADf_j, D\varphi \rangle dX \leq 0$. A similar computation shows
that $\L f = 0$ in $D_\delta$.
\par
Let us turn our attention to the $W^{1,p}$-weak lower
semicontinuity of the functional $\J_\lambda$. Firstly, the volume
penalty term of the functional $\J_\lambda$ is indeed weak lower
semicontinuous, since, up to a subsequence, $f_j(X) \to f(X)$ for
a.e. $X \in D^C$. Thus, by Fatou's Lemma
$$
    | \{ f > 0 \} | \le \liminf\limits_{j\to\infty} | \{ f_j > 0 \} |.
$$
Since, the penalty factor $\varrho_\lambda$ is non-decreasing and
continuous, there holds
$$
    \varrho_\lambda\left (  | \{ f > 0 \} | \right ) \le
    \liminf\limits_{j\to\infty}  \varrho_\lambda\left (  | \{ f_j > 0 \} | \right
    ),
$$
as desired. We now focus our attention on the functional $\J(v) =
\int_{\partial D} \Gamma(X, \partial_\A v) d\H$. As in the
Calculus of Variations, in order to establish the $W^{1,p}$-weak
lower semicontinuity of $\J$, we shall explore the convexity
assumption on $\Gamma(X,\cdot)$. Indeed, we start by analyzing
functional with piecewise linear potential, i.e., functionals with
this particular profile:
\begin{equation}\label{semi-continuity - functional form}
    \mathfrak{F}_m (v) := \int_{\partial D} F_m(X,\partial_\A v) d\H,
\end{equation}
where $F_m$ is of the form
\begin{equation}\label{semi-continuity - Fm form}
    F_m(X,t) = \max\limits_{1\le k \le m} \left \{ B_k(X) t + C_k(X) \right
    \}, \quad B_k, ~ C_k \in C(\partial D).
\end{equation}
We then label, for each $k = 1, 2, \cdots, m$, the sets
$$
    \mathcal{D}_k(f) :=\{ X \in \partial D \suchthat F_m (X, \partial_\A f(X) ) =
    B_k(X) \partial_\A f(X) + C_k(X) \}.
$$
Thus $\partial D = \bigcup\limits_{k=1}^{m} \mathcal{D}_k(f)$, and
we may assume that $\mathcal{D}_k(f) \cap \mathcal{D}_i(f) =
\emptyset$, whenever $k \not = i$. Also, recall that $\L f_j$ and
$\L f$ define Radon measures in $D^C$, and since $f_j
\rightharpoonup f$ in $W^{1,p}$, by standard elliptic estimates,
we have
$$
    \L f_j \stackrel{\star}{\rightharpoonup} \L f,
$$
in the sense of Radon measures. Therefore, using a representation
as in (\ref{meaure mu is normal bdry int}), we obtain that, for
any continuous function $\zeta \in C(\partial D)$,
$$
    \int_{\partial D} \zeta(X) \partial_\A f(X)  d \H \le
    \liminf\limits_{j\to \infty} \int_{\partial D}  \zeta(X) \partial_\A f_j(X)  d \H.
$$
With the above at hands, we estimate
$$
\begin{array}{lll}
    \mathfrak{F}_m(f) & = & \displaystyle\sum\limits_{k=1}^m \int_{\mathcal{D}_k(f)}
      \left \{ B_k(X)\partial_\A f + C_k(X) \right \} d\H \\
    & \le & \liminf\limits_{j\to \infty} \displaystyle\sum\limits_{k=1}^m \int_{\mathcal{D}_j(f)}
    \left \{ B_k(X) \partial_\A f_j  + C_k(X) \right \} d\H  \\
    & \le & \liminf\limits_{j\to \infty} \mathfrak{F}_m(f_j),
\end{array}
$$
In other words, we have proven functionals as in
(\ref{semi-continuity - functional form}) are $W^{1,p}$-weak lower
semicontinuous. Finally, under the assumption that
$\Gamma(X,\cdot)$ is convex we know that for each $X\in
\partial D$ there exits a sequence of functions $F_m(X,t)$ as in
(\ref{semi-continuity - Fm form}) such that, for any $t$,
\begin{equation}\label{semi-continuity 2}
    \Gamma(X,t) = \lim\limits_{m \to \infty} F_m(X,t).
\end{equation}
As a combination of (\ref{semi-continuity 2}) and the
$W^{1,p}$-weak lower semicontinuity of each $\mathfrak{F}_m$, the
Lemma follows.
\end{proof}
The results proven in Proposition \ref{Energ Estimate} and in
Lemma \ref{semi-continuity} are important piece of information
towards establishing the existence of an optimal shape for problem
(\ref{Penalized Min Problem}); however, at this precise stage,
those are not enough. We would like to invite the readers to make
a small pause in order to appreciate the intrinsic difficulty
involved in proving the existence of a minimal configuration to
the penalized problem (\ref{Penalized Min Problem}).
\par
Following the natural scheme, one considers a minimizing sequence,
$\Omega_j$, to the functional $\J_\lambda$, i.e.,
$$
    \J_\lambda(\Omega_j) \stackrel{j\to \infty}{\longrightarrow}
    \min\limits_{\Omega \supset D} \J_\lambda.
$$
If $u_j$ denotes the $\A$-potential associated to the
configuration $\Omega_j$, it follows from Proposition \ref{Energ
Estimate} that, up to a subsequence, $u_j$ converges weakly and
almost everywhere to a function $u \in W^{1,p}(D^C)$ which is
non-negative. As a consequence of Lemma \ref{semi-continuity}, we
have that $\L u \ge 0 $. Furthermore,
$$
    \int_{\partial D} \Gamma(X, \partial_\A u)d\H + \varrho_\lambda \left (
    |\{ u > 0 \} | \right ) \le \min\limits_{\Omega \supset D}
    \J_\lambda.
$$
Therefore, a natural candidate for an optimal shape to problem
(\ref{Penalized Min Problem}) is
$$
    \Omega := \left \{X \in \R \setminus D \suchthat u(X) > 0 \right \}.
$$
However, with the information we have so far, it is not possible
to guarantee that $(\Omega, u)$ is an admissible pair, i.e., that
$u$ is the $\A$-potential associated to $\Omega$, or equivalently
that
$$
    \L u = 0 \textrm{ in } \Omega.
$$
In fact, it is not true, in general, that if an ordinary sequence
of functions $u_j$, satisfying $\Delta u_j = 0$ in $\{u_j > 0\}$,
converges weakly in $H^1$ to $u$, then $\Delta u = 0$ in $\{u >
0\}$. As a general comment, the above described difficulty is one
of the features that makes problems with varying domains (free
boundary problems) notably more delicate.

\section{Existence of minimizer to problem
$(\mathfrak{P}_\lambda)$} \label{SECTION - Existence of minimizer
to problem Penalized}
Well, the scheme presented at the end of the previous section is
not pointless: since our sequence is converging to a special
configuration, namely a minimizer for the functional $\J_\lambda$,
we should keep the hope that this strong additional ingredient
will assure that in fact $\L u = 0$ in $\{u > 0 \}$. In this
section, we will carry this delicate analysis out, which will
ultimately allow us to conclude problem (\ref{Weak Min Problem})
has always a minimizer. As we will see, even the weak formulation
of the penalty version of our primary goal presents rather
delicate mathematical issues. This is due in part to the adverse
environment generated by the non-linear and degeneracy features of
$\A$, and in part to the non-local structure of the problem. The
latter makes local perturbations inefficient, and thus  more
creativity is needed to furnish appropriate competing
configurations.
\par
As for our first result towards the existence of a minimizer for
problem (\ref{Weak Min Problem}),  we will provide an a-priori
estimate on the distance from the free boundary to the fixed
boundary. This is an important supporting result as it allows to
seek for minimizers in a more suitable class of configurations.
\par
However, in order to accomplish such a result, we initially need
to study an auxiliary free boundary problem in the spirt of
\cite{AC}, which we present now.
\begin{theorem} \label{AC for A-potentials}Let $\mathcal{O}$ be a domain in
$\R^n$ and $\psi \colon \mathcal{O} \to \R$ a nonnegative
function. Let $\A$ be a $p$-degenerate elliptic map and assume $\A
(\cdot, \mathcal{O}) \in C^\epsilon$ for all $\mathcal{O}$. Then,
for any constant $\tau >0$, there exists a minimizer $v=v_\tau$ to
the problem
$$
    \textrm{Minimize } \quad \left \{
        {E}_\tau(f) := \displaystyle \int\limits_{\mathcal{O}}
        \left \{ \langle \A(X,Df), Df \rangle + \tau \chi_{\{f > 0\}} \right \} dX
        \suchthat f \in W^{1,p}(\mathcal{O}),~
        f\big |_{\partial \mathcal{O}} = \psi
         \right \}.
$$
Furthermore, $v$ is nonnegative, Lipschitz continuous and
nondegenerate away from the free boundary $\partial \{ v_\tau > 0
\}$.
\end{theorem}

With the free boundary technology available nowadays, it is not
hard to establish the existence as well as optimal regularity and
nondegeneracy of a minimizer to the above problem. Basically there
are two procedures that lead to these results: one can directly
approach the minimization problem, by mixing the strategy as in
\cite{AC} and \cite{DP}. Another charming and fruitful strategy is
to employ a regularizing technique method, basically by mixing the
estimates in \cite{MT} or \cite{T1} and \cite{DPS}, see also
\cite{K}. Mathematically the latter is described as follows:
choose your favorite nonnegative bounded real function $\beta$,
such that $\supp \beta = [0,1]$ and, say, $\int_0^1
\beta(\zeta)d\zeta =1$. For each $\varepsilon > 0$ define
$$
    \beta_\varepsilon(t) := \dfrac{1}{\varepsilon} \beta \left (
    \dfrac{t}{\varepsilon} \right ),
$$
and finally put $B_\varepsilon(s) := \int_0^s
\beta_\varepsilon(\zeta) d\zeta$. The $\varepsilon$ regularizing
problem then becomes
\begin{equation}\label{Regularizing Min Problem in AC}
    \textrm{Minimize } \quad \left \{
        {E}^\varepsilon_\tau(f) := \displaystyle \int\limits_{\mathcal{O}}
        \left \{ \langle \A(X,Df), Df \rangle + {\tau} B_\varepsilon(f) \right \} dX
        \suchthat f \in W^{1,p}(\mathcal{O}),~
        f\big |_{\partial \mathcal{O}} = \psi
         \right \},
\end{equation}
The existence of minimizers $v^\varepsilon$ of (\ref{Regularizing
Min Problem in AC}) is standard. One then proves Lipschitz
regularity and nondegeneracy for $v^\varepsilon$, uniform in
$\varepsilon$. By letting $\varepsilon \searrow 0$, up to a
subsequence, $v^\varepsilon$ will converge to a locally Lipschitz
function $v$ that is a minimizer of $E_\tau$.  We omit the details
of the proof of Theorem \ref{AC for A-potentials}.
\begin{proposition}\label{Min is positive around D} There exists a
positive constant $\gamma > 0$, depending  only  on dimension,
$\lambda$, $\partial D$, $\Gamma$ and $\varphi$ such that any
(possible) minimizer $u^\star$ of problem (\ref{Weak Min Problem})
satisfies
$$
    D_\gamma := \left \{ X \in D^C \suchthat \dist
    (X, \partial D) \le \gamma \right \} \subset \{ u^\star > 0
    \}.
$$
\end{proposition}
\begin{proof}
Let $P \in \partial D$ be fixed and $B = B_r(Y) \subset D$ satisfy
$$
    \overline{B}\cap \partial D = \{P\}.
$$
By a compactness argument on $\partial D$, we can select an $r <
5\delta_0$, where $\delta_0$ is the universal number from
(\ref{Continuity of A close to D}), such that the above holds for
a.e. $P \in \partial D$. In view of Theorem \ref{AC for
A-potentials}, there exists a minimization, $v = v(\tau)$ to
\begin{equation}\label{AC for A-potentials}
    \textrm{Min } \quad \left \{ \begin{array}{c}
        E_\tau(f) := \displaystyle \int\limits_{5B \setminus
        B} \left \{ \langle \A(X,Df), Df \rangle + \tau \chi_{\{f > 0\}} \right \} dX
        \suchthat f \in W^{1,p}(5B\setminus B), \\
        f\big |_{\partial 5B} = 0, \quad \textrm{ and } \quad
        f\big |_{\partial B} = \inf\limits_{\partial D} \varphi
        \end{array}
    \right \}.
\end{equation}
Here $\tau > 0$ is a constant to be chosen later.  For future
reference, let us label the following the sets
$$
    \Theta := \{ X \in D^C\cap 5B \suchthat v(X) > u(X) \} \quad \textrm{and} \quad
    \mathcal{O} := \{ X \in D^C \cap 5B \suchthat v(X) > 0 \}.
$$
It is important to keep in mind that, from the properties of $v$,
we can ensure that there exist constants $\theta, \hat{\delta}
> 0$, depending only on $\A$, $\tau$, $\partial D$ and $\inf \varphi$ such
that
\begin{equation}\label{Prop Min is positive around D - Key Ingredient}
    |\mathcal{O} \cap D^C| > \theta(\tau), \quad \textrm{ and } \quad
    \dist \left ( P, \left (\partial \mathcal{O} \cap D^C\right) \right) > \hat{\delta}(\tau).
\end{equation}
We now define the function $\mathfrak{m} \colon 5B \setminus B \to
\R_{+}$ as
$$
    \mathfrak{m}(X) := \left \{
        \begin{array}{cll}
            v(X) & \textrm{in}& (D \setminus B) \cap 5B \\
            \min\{u(X), v(X) \} & \textrm{in}& D^C \cap 5B,
        \end{array}
    \right.
$$
Since $\mathfrak{m}$ competes with $v$ in the minimization problem
(\ref{AC for A-potentials}), we have $E_\tau(v) \le
E_\tau(\mathfrak{m})$. Hence, the following inequality holds
\begin{equation}\label{Prop Min is positive around D - 1st ineq}
        \displaystyle \int_\Theta \langle \A(X,Du),Du \rangle dX -
        \displaystyle \int_\Theta \langle \A(X,Dv),Dv \rangle
        dX \ge \tau \left \{ \Leb \left ( \mathcal{O} \right ) -
        \Leb \left (  \{u > 0 \} \cap \mathcal{O} \right ) \right \}
\end{equation}
Our strategy now is to obtain a competing inequality to (\ref{Prop
Min is positive around D - 1st ineq}). To this end, let us
consider the function $\mathfrak{M} \colon D^C \to \R_{+}$ defined
as
$$
    \mathfrak{M}(X) := \max \{ v(X), u(X) \},
$$
and compare it with $u$ in terms of the minimization problem
(\ref{Weak Min Problem}). Using the minimality of $u$, we obtain
\begin{equation}\label{Prop Min is positive around D 2nd ineq}
    \begin{array}{lll}
        \varrho_\lambda \Big ( \Leb \left ( \mathcal{O} \right ) + \Leb \left (  \{u > 0 \}
        \right ) - \Leb \left (  \{ u > 0 \} \cap \mathcal{O} \right ) \Big )
        &-&
         \varrho_\lambda \Big ( \Leb \left ( \{u > 0 \} \right ) \Big ) \\
         &\ge&
        \displaystyle \int_{\partial D} \Gamma(X, \partial_\A
        u) - \Gamma(X, \partial_\A \mathfrak{M}) d\H.
    \end{array}
\end{equation}
From properties 1 and 2 of $\Gamma$, and the Lipschitz continuity
of the penalty term $\varrho_\lambda$, we conclude from (\ref{Prop
Min is positive around D 2nd ineq}) that there exists a small
constant $\alpha_0 = \alpha_0(\partial D, \Gamma)$ such that
\begin{equation}\label{Prop Min is positive around D 3nd ineq}
        \dfrac{\lambda}{\alpha_0} \big ( \Leb \left ( \mathcal{O} \right )
        - \Leb \left ( \{ u > 0 \} \cap \mathcal{O} \right ) \big ) \ge
        \displaystyle \int_{\partial D} \left \{ \partial_\A u - \partial_\A \mathfrak{M}  \right \} d\H.
\end{equation}
Applying the Divergence Theorem (see the representation in
(\ref{weak Div Thm})) and taking into account that $v(X)\L v(X) =
0$ a.e., we obtain
\begin{equation}\label{Prop Min is positive around D 4nd ineq}
        \displaystyle \int_{\partial D} \left \{ \partial_\A u - \partial_\A \mathfrak{M}  \right \} d\H
        \ge \dfrac{1}{\sup\limits_{\partial D} \varphi } \displaystyle \int_\Theta \langle ADu,Du \rangle -
        \langle ADv,Dv \rangle dX.
\end{equation}
As a combination of (\ref{Prop Min is positive around D - 1st
ineq}), (\ref{Prop Min is positive around D 3nd ineq}) and
(\ref{Prop Min is positive around D 4nd ineq}) we deduce that
$$
    \sup\limits_{\partial D} \varphi \cdot
    \dfrac{\lambda}{\alpha_0} \Big [ \Leb \left (\mathcal{O} \right )
    - \Leb \left ( \{ u > 0 \} \cap \mathcal{O} \right ) \Big ]
    \ge \tau \Big [ \Leb \left (\mathcal{O} \right )
    - \Leb \left ( \{ u > 0 \} \cap \mathcal{O} \right ) \Big ].
$$
Thus, if $\tau$ is chosen big enough, depending only upon
dimension, $\A$, $\partial D$ and $\varphi$, there must be the
case that
$$
   \mathcal{O} \subset \{ u^\star > 0 \}.
$$
This together with (\ref{Prop Min is positive around D - Key
Ingredient}) ultimately finishes the proof of the Proposition.
\end{proof}
In order to advance in our analysis, we need another related free
boundary problem: an $\A$-obstacle type problem, which again, with
the free boundary technology available, is easy accomplished and
therefore we omit the details.
\begin{theorem}\label{obstable problem} Let $\mathcal{M}$ be a measurable set in
$D^C$. There exists a unique function $\mathfrak{b}$, solution to
the following obstacle-type problem:
$$
    \textrm{Min } \left \{ \int_{D^C} \langle \A(X,Df),Df \rangle dX
    \suchthat f \in W^{1,p}(D^C)~ f = \varphi \textrm{ on } \partial D
    \textrm{ and } f \le 0 \textrm{ in } \mathcal{M}
    \right \}.
$$
Furthermore, $\sup \varphi \ge \mathfrak{b}\ge 0$, $\L
\mathfrak{b} = 0$ in $\{ \mathfrak{b} > 0 \}$ and $\int
\mathfrak{b} \L \mathfrak{b} dX = 0$.
\end{theorem}
We now can state and proof our main theorem concerning the
existence of an optimal configuration to weak formulations of
problem (\ref{Minimization Problem}), namely problems
$(\mathfrak{P}^\textrm{weak}_\lambda)$ and
$(\mathfrak{P}_\lambda)$.
\begin{theorem} \label{Exist for penalized problem}
There exists an optimal configuration $\Omega_\lambda^\star$ to
 problem (\ref{Penalized Min Problem}) (the penalized problem $(\mathfrak{P}_\lambda)$).
Furthermore, for a universal modulus of continuity $\sigma$, the
$\A$-potential associated to $\Omega_\lambda^\star$,
$u^\star_\lambda$, is $\sigma$-continuous in $D^C$ and
$\|u^\star_\lambda\|_{C^{\sigma}} \lesssim K(\lambda, D, \varphi,
\Gamma, \A)$.
\end{theorem}
\begin{proof}
Before starting the proof, let us explain its strategy. We will
initially establish the existence of a minimizer to a very weak
formulation to  problem (\ref{Penalized Min Problem}). Afterwards
we  ``regularize" the minimizer we have found via a stabilization
phenomenon. Here are the details: Lemma \ref{semi-continuity}
assures, for each $0<\delta <\!\!< 1$, the existence of a function
$u^\delta_\lambda \in \mathcal{V}(\delta)$, satisfying
$$
    \J_\lambda(u^\delta_\lambda) = \min\limits_{V(\delta)}
    \J_\lambda.
$$
Furthermore, by  noticing that the same computation employed in
the proof of Theorem \ref{Min is positive around D} applies if we
restrict ourselves to configurations in $\mathcal{V}(\delta)$, we
know that
$$
    D_\gamma \subset \{ u^\delta_\lambda > 0 \}, \quad \forall
    \delta >0.
$$
Let $B = B_r(X_0)$ be a fixed ball in $D^C$ and $\mathfrak{b}$ be
the solution provided by Theorem \ref{obstable problem} to
\begin{equation}\label{obs problem - uniform continuity}
    \textrm{Min } \left \{ \int_{D^C} \langle \A(X,Df),Df \rangle dX
    \suchthat f \in W^{1,p}(D^C)~ f = \varphi \textrm{ on } \partial D
    \textrm{ and } f \le 0 \textrm{ in } \{ u_\lambda^\delta = 0 \} \setminus B
    \right \}.
\end{equation}
We also consider $\mathfrak{h}$ to be the $\A$-harmonic function
in $B$ that agrees with $u_\lambda^\delta$ on $B^C$. It is
standard to verify that
\begin{equation}\label{Existence Thm Eq 01}
    0\le u \le \mathfrak{b} \le \mathfrak{h} \le \sup \varphi.
\end{equation}
As before, (more precisely, as in the proof of Proposition
\ref{Min is positive around D}) taking into account that $\int
\mathfrak{b} \mathfrak{L} \mathfrak{b} dX = 0$, we find
\begin{equation} \label{uniform continuity Eq01}
    \int_{\partial D} \Gamma (X, \partial_\A u) - \Gamma (X, \partial_\A
    \mathfrak{b}) \ge c_1 \left ( \int_{D^C} \langle \A(X,Du_\lambda^\delta), Du_\lambda^\delta
    \rangle dX -
    \int_{D^C} \langle \A(X, D\mathfrak{b}), D \mathfrak{b}\rangle dX \right
    ),
\end{equation}
for a universal positive constant $c_1 > 0$. However,
$\mathfrak{h}$ competes with $\mathfrak{b}$ in the obstacle
problem (\ref{obs problem - uniform continuity}), thus,
(\ref{uniform continuity Eq01}) becomes
\begin{equation} \label{uniform continuity Eq02}
     \int_{\partial D} \Gamma (X, \partial_\A u) - \Gamma (X, \partial_\A
    \mathfrak{b}) \ge c_1 \left ( \int_{D^C} \langle A(X,Du_\lambda^\delta), Du_\lambda^\delta
    \rangle dX -
    \int_{D^C} \langle \A(X, D\mathfrak{h}), D \mathfrak{h}\rangle dX \right ).
\end{equation}
For the moment, let us assume $p \ge 2$. If we take into account
Lemma \ref{p harmonic functions with same bondary data}, we can
enhance the estimate by below in (\ref{uniform continuity Eq02})
as
\begin{equation} \label{uniform continuity Eq02.5}
     \int_{\partial D} \Gamma (X, \partial_\A u) - \Gamma (X, \partial_\A
    \mathfrak{b}) \ge c_2 \left ( \int_{D^C} \left |
    \nabla \left (u_\lambda^\delta - \mathfrak{h}\right )(X) \right|^p  dX
\right ),
\end{equation}
for an appropriate positive but small constant $c_2$. Our next
step is to compare $u_\lambda^\delta$ and $\mathfrak{b}$ in terms
of the functional $\J_\lambda$. By doing so, in view of
(\ref{uniform continuity Eq02.5}), we obtain
\begin{equation} \label{uniform continuity Eq03}
     \lambda \Leb \left ( \{ X \in B_r(X_0) \suchthat u_\lambda^\delta(X) = 0 \} \right ) \ge
     c_3  \left ( \int_{D^C} \left |
    \nabla \left (u_\lambda^\delta - \mathfrak{h}\right )(X) \right|^p  dX \right ),
\end{equation}
for another constant $c_3>0$, depending on dimension, $\A$, $\sup
\varphi$, and $\Gamma$. If $1 < p \leq 2$, we obtain
\begin{equation} \label{uniform continuity Eq04}
     \left [\lambda \Leb \left (  \{ X \in B_r(X_0) \suchthat u_\lambda^\delta(X) = 0 \}
     \right )\right ]^{p/2} \times
     \left [ \int_{D^C} |\nabla u_\lambda^\delta|^p dX \right ]^{1 - \frac{p}{2}} \ge
     c_3 \left ( \int_{D^C} \left |
    \nabla \left (u_\lambda^\delta - \mathfrak{h}\right )(X) \right|^p  dX \right ).
\end{equation}
In any case, our conclusion is that if $B_r(X_0) \subset
D_\gamma$, then $|\{ X \in B_r(X_0) \suchthat u_\lambda^\delta(X)
= 0 \} | = 0$ and consequently, from either (\ref{uniform
continuity Eq03}) or (\ref{uniform continuity Eq04}),
$u_\lambda^\delta$ is $\A$-harmonic there. Of course $
    \J(u_\lambda^{\delta_1}) \le  \J(u_\lambda^{\delta_2}),
$ provided $\delta_1 \le \delta_2$. However, from the fact that
$\L u_\lambda^\delta = 0$ in $D_\gamma$ we have a much stronger
conclusion:
$$
    \J(u_\lambda^{\delta_1}) =  \J(u_\lambda^{\delta_2}),
$$
whenever $\delta_1, ~ \delta_2 \le \gamma$. We have proven the
existence of a minimizer $u_\lambda^\star$ to
$(\mathfrak{P}_\lambda^\mathrm{weak})$, that is, problem
(\ref{Weak Min Problem}).
\par
Our next step is now to prove that $\Omega^\star := \{
u_\lambda^\star
> 0 \}$ is a minimizer to problem (\ref{Penalized Min Problem}). For
that, we have to show
$$
    \L u_\lambda^\star = 0 \textrm{ in } \Omega^\star.
$$
Well, but again it is a standard argument to show from either
(\ref{uniform continuity Eq03}) or (\ref{uniform continuity Eq04})
that $u_\lambda^\star$ belongs to an appropriate De Giorgi's class
(recall $\mathfrak{h}$ is H\"older continuous by elliptic
estimates). Therefore, there indeed exists a modulus of continuity
$\sigma$ ($\sigma (t) = |t|^\alpha$, for some $\alpha>0$), such
that
$$
  \left |  u(X) - u(Y) \right | \le C \lambda \sigma(|X-Y|).
$$
In order to prove that $\mathfrak{L} u = 0$ in $\{u > 0\}$, we
argue as follows: let $X_0 \in \{u > 0\}$ be a generic point. By
the continuity of $u$, there exists an $r_0 >0$ such that
$B_{r_0}(X_0) \subset \{ u > 0 \}$. Therefore, in view of
(\ref{uniform continuity Eq03}) or (\ref{uniform continuity
Eq04}), we conclude, as before that
$$
    u = \mathfrak h \textrm{ in } B_{r_0}(X_0),
$$
and the Theorem is finally proven.
\end{proof}

\section{Existence of an optimal shape to problem (\ref{Minimization Problem}) in low dimensions} \label{SECTION - Existence up to dimension p}

In this section, upon a technical restriction on the dimension, we
will show that the original volume constrained problem
(\ref{Minimization Problem}) admits an optimal configuration. The
theory that addresses the existence of an optimal design for
problem (\ref{Minimization Problem}) in all dimensions will be
developed in section \ref{SECTION - Existence theory in any
dimension}.
\par
Our strategy is based on a limiting analysis on the penalized
problem (\ref{Penalized Min Problem}). For that, we initially need
a simple lemma.
\begin{lemma}\label{control on gradient} There exists a constant $C>0$, depending on
$\A$, $\Gamma$, $D$ and $\varphi$, but independent of $\lambda$,
such that if $u^\star_\lambda$ is the $\A$-potential associated to
an optimal shape $\Omega^\star_\lambda$ for problem
(\ref{Penalized Min Problem}), then
$$
    \int_{D^C} |\nabla u_\lambda^\star (X)|^p dX < C.
$$
\end{lemma}
\begin{proof}
Let $\mathcal{O}$ be your favorite smooth configuration
surrounding $D$ that satisfies
$$
    \Leb \left ( \mathcal{O} \setminus D \right ) = \iota,
$$
and let $\omega$ be its $\A$-potential, i.e., the $\A$-harmonic
function in $\mathcal{O} \setminus D$ taking $\varphi$ and $0$ as
boundary data on $\partial D$ and $\partial \mathcal{O}$
respectively. By the minimality property of
$\Omega^\star_\lambda$, we know
\begin{equation}\label{bound grad Eq01}
    \begin{array}{lll}
        \displaystyle \int_{\partial D} \Gamma(X, \partial_\A u^\star_\lambda) d\H
        &\le& \J_\lambda (\Omega^\star_\lambda) \\
        &\le& \J_\lambda (\mathcal{O}) \\
        & = & \displaystyle \int_{\partial D} \Gamma(X, \partial_\A \omega)
        d\H \\
        &=& \overline{C}_0,
    \end{array}
\end{equation}
where $\overline{C}_0$ is universal, as it depends only on you
choice for $\mathcal{O}$. On the other hand, using the results and
notations of Lemma \ref{Lf is a Radon Measure}, we have
\begin{equation}\label{bound grad Eq02}
    \begin{array}{lll}
        \displaystyle \int_{\R^n \setminus D} \langle \A(X,Du_\lambda^\star), Du_\lambda^\star
        \rangle dX &=& \displaystyle \int_{\R^n \setminus D}
        u_\lambda^\star(X) d\mu_{u_\lambda^\star}(X) \\
        &\le& \sup\limits_{\partial D} \varphi
        \cdot \mu_{u_\lambda^\star}(\R^n \setminus D) \\
        &=& \sup\limits_{\partial D} \varphi \cdot
        \displaystyle \int_{\partial D} \partial_\A
        u^\star_\lambda(S) d\mathcal{H}^{n-1}(S).
    \end{array}
\end{equation}
From ellipticity and (\ref{bound grad Eq02}), we conclude
\begin{equation}\label{bound grad Eq03}
    \underline{c}_1 \int_{D^C} |\nabla u^\star_\lambda(X)|^p dX \le
    \dfrac{1}{\mathcal{H}^{n-1}(\partial D)} \int_{\partial D} \partial_\A
    u^\star_\lambda(S)d\mathcal{H}^{n-1}(S),
\end{equation}
where $\underline{c}_1$ is a positive number that depends on $\A$,
$\varphi$ and $D$. Now, for each $Y \in \partial D$ fixed, we
obtain from (\ref{bound grad Eq03})
\begin{equation}\label{bound grad Eq04}
    \begin{array}{lll}
        \Gamma \left (Y, \underline{c}_1  \displaystyle\int_{D^C} |\nabla u^\star_\lambda(X)|^p dX
        \right ) &\le&
        \Gamma \left (Y,  \dfrac{1}{\mathcal{H}^{n-1}(\partial D)}
         \displaystyle \int_{\partial D} \partial_\A
        u^\star_\lambda(S)d\mathcal{H}^{n-1}(S) \right ) \\
        &\le & \dfrac{1}{\mathcal{H}^{n-1}(\partial D)}  \displaystyle \int_{\partial D}
        \Gamma \left (Y,  \partial_\A
        u^\star_\lambda(S) \right ) d\mathcal{H}^{n-1}(S).
    \end{array}
\end{equation}
In the last inequality we have used Jensen's Theorem. If we
integrate (\ref{bound grad Eq04}) with respect to $Y$ over
$\partial D$, we reach the following conclusion
\begin{equation}\label{bound grad Eq05}
    \int_{\partial D}
    \Gamma \left (Y, \underline{c}_1  \displaystyle\int_{D^C} |\nabla u^\star_\lambda(X)|^p dX
        \right ) d\mathcal{H}^{n-1}(Y) \le \overline{C}_2
         \displaystyle \int_{\partial D} \Gamma(X, \partial_\A u^\star_\lambda)
         d\H,
\end{equation}
where $\overline{C}_2$ depends only on $\partial D$ and the
non-linearity $\Gamma$. Finally, if we combine (\ref{bound grad
Eq01}), (\ref{bound grad Eq05}) and (\ref{coersivity}), we deduce
that there must exist a constant $C>0$ depending only on $\A$,
$\Gamma$, $D$ and $\varphi$, such that
\begin{equation}\label{bound grad Eq08}
        \int_{D^C} |\nabla u_\lambda^\star(X)|^p dX \le
        C,
\end{equation}
which is precisely the thesis of the Lemma.
\end{proof}
\begin{theorem} \label{existence of a minimum up to dim p}
Assume the dimension $n$ is less than $p$. Then there exists an
optimal configuration $\Omega^\star$ to problem (\ref{Minimization
Problem}).
\end{theorem}
\begin{proof}
Because of Lemma (\ref{control on gradient}), up to a subsequence,
we can assume $u_\lambda$ converges, as $\lambda \to \infty$,
weakly in $W^{1,p}(D^C)$ to a function $u^\star$. Furthermore,
since we have assumed $n < p$, it follows by the classical Sobolev
Imbedding (see, for instance, \cite{A}), that  passing to another
subsequence if necessary, we can further assume that $u_\lambda$
converges locally uniformly to $u^\star$ in $\R^n \setminus D$ and
thus, $u^\star$ is continuous in $D^C$. We claim that
$$
    \L u^\star = 0 \textrm{ in } \Omega^\star := \{ X \in D^C
    \suchthat u^\star(X) > 0 \}.
$$
Indeed, let $X_0 \in \Omega^\star$ be an arbitrary point in the
set of positivity of $u^\star$, say $u^\star (X_0) = \delta_0 >
0$. By continuity, there exists an $r_0 > 0$ such that
$$
    u^\star(X) > \dfrac{\delta_0}{3} \textrm{ in } B_{r_0}(X_0).
$$
Since $u^\star_\lambda$ converges uniformly to $u^\star$ in
$B_{r_0}(X_0)$, there exists a  $\lambda_0$ large enough, such
that
$$
    u^\star_\lambda(X) > \dfrac{\delta_0}{7} \textrm{ in }
    B_{r_0}(X_0), \quad \forall \lambda > \lambda_0.
$$
However, we have proven that $\L u_\lambda^\star = 0$ in $\{
u^\star_\lambda > 0 \}$. Therefore, for $\lambda$ large enough,
each $u^\star_\lambda$ is $\A$-harmonic in $B_{r_0}(X_0)$. Thus,
as argued in the proof of Lemma \ref{semi-continuity}, we in fact
conclude $u^\star$ is $\A$-harmonic in its set of positivity and
the first claim is proven.
\par
Notice furthermore that, in view of Proposition \ref{Min is
positive around D},
$$
    \dist (\partial D, \partial \Omega^\star) > \gamma,
$$
for some $\gamma > 0$. From inequality (\ref{bound grad Eq01}), we
have, in particular, that
$$
    \lambda \left ( \Leb \left ( \Omega^\star_\lambda \setminus D \right ) - \iota
    \right )^{+} \le \overline{C}_0,
$$
for a universal constant $\overline{C}_0$. Thus, using Fatou's
Lemma we see that
$$
    \begin{array}{lll}
        \left ( \Leb \left ( \Omega^\star \setminus D \right ) - \iota \right )^{+}
        &\le& \liminf\limits_{\lambda \to \infty}
        \left ( \Leb \left ( \Omega^\star_\lambda \setminus D \right ) - \iota \right
        )^{+} \\
        & = & 0.
    \end{array}
$$
That is, our candidate to an optimal design for problem
(\ref{Minimization Problem}), $\Omega^\star$, does satisfy
$$
    \Leb \left ( \Omega^\star \setminus D \right ) \le \iota,
$$
so it competes in problem (\ref{Minimization Problem}). Our final
step is to show that in fact $\Omega^\star$ is an optimal
configuration for problem (\ref{Minimization Problem}). For that,
let $\mathfrak{C}$ be any competing configuration for problem
(\ref{Minimization Problem}), i.e., $\Leb ( \mathfrak{C} \setminus
D )\le \iota$, and $v$ its $\A$-potential, that is, $v$ satisfies
$$
    \L v = 0 \textrm{ in } \mathfrak{C} \setminus D, \quad v =
    \varphi \textrm{ on } \partial D, \quad v = 0 \textrm{ on }
    \partial \mathfrak{C}.
$$
In particular $\mathfrak{C}$ competes with $u_\lambda^\star$ in
$(\mathfrak{P}_\lambda)$, problem (\ref{Penalized Min Problem});
therefore,
$$
    \begin{array}{lll}
        \J(\mathfrak{C}) & := &
        \displaystyle \int_{\partial D} \Gamma \left ( X, \partial_\A v(X) \right )
        d\H \\
        & = & \J_\lambda(\mathfrak{C}) \\
        &\ge& \J_\lambda(\Omega_\lambda^\star) \\
        & \ge &
        \displaystyle\int_{\partial D} \Gamma \left ( X, \partial_\A u_\lambda^\star (X) \right )
        d\H \\
        & \ge & \displaystyle\int_{\partial D} \Gamma \left ( X, \partial_\A u^\star (X) \right )
        d\H + O(1),
    \end{array}
$$
because of the weak lower semicontinuity feature of $\J$ proven in
Lemma \ref{semi-continuity}. Finally if we let $\lambda \to
\infty$ in the above chain of inequalities, the Theorem is proven.
\end{proof}
It is worth to point out that Theorem \ref{existence of a minimum
up to dim p} gives the existence of an optimal configuration to
problem (\ref{Minimization Problem}) with no regularity whatsoever
on the medium. That is, up to this point of the project, the
operator $\A$ has been a general bounded measurable degenerated
elliptic map. However, it turns out that in order to advance on
the study of existence of optimal shapes for problem
(\ref{Minimization Problem}), with no restriction on the
dimension, some extra information is needed to perform appropriate
perturbations on the optimal designs $\Omega_\lambda^\star$. This
will be the contents of the next two sections.

\section{Continuous medium and fine weak geometric properties of the free boundary}
\label{SECTION - Weak Geometric Properties of the free boundary}

In this section we will prove that the free boundary, $\partial
\Omega^\star_\lambda$ enjoys the appropriate weak geometry. This
feature will allow us to produce geometric-measures perturbations
that will ultimately lead us to conclude that, if the penalty term
$\lambda$ is too large, but still finite, then
$\Omega_\lambda^\star$, in fact, obey $  \Leb (
\Omega_\lambda^\star \setminus D ) \le \iota. $ The latter will be
carried out in section \ref{SECTION - Existence theory in any
dimension}.
\par
As highlighted in the last paragraph of the previous section, in
order to accomplish a deeper understanding on the free boundary
$\partial \Omega^\star_\lambda$, we will need to enforce a mild
continuity assumption on the medium. Thus, hereafter, unless
otherwise stated, we shall assume
 that for some $\epsilon > 0$, the map
\begin{equation} \label{continuity of the medium}
    X \mapsto \A(X, \xi) \in C^\epsilon(\R^n \setminus D), \quad \forall \xi \in
    \R^n.
\end{equation}
Mathematically, condition (\ref{continuity of the medium}) enables
$C^{1,\alpha}$ elliptic estimates for solutions to
$$
    \L \psi = 0
$$
and, at least equally important, it unlocks the Hopf's maximum
principle for $\A$-harmonic functions.
\par
As for our first theorem in this section, we will obtain optimal
regularity for $\A$-potentials $u^\star_\lambda$ associated to
optimal configurations $\Omega_\lambda^\star$ of Problem
$(\mathfrak{P}_\lambda)$, that is, Problem (\ref{Penalized Min
Problem}). Notice that inside $\Omega_\lambda^\star$, the function
$u^\star_\lambda$ satisfies $\L u^\star_\lambda$; therefore, it is
locally $C^{1,\alpha}$ smooth. However, from the Hopf's maximum
principle, $u^\star_\lambda$ reaches the free boundary with a
positive slope, thus $\nabla u^\star_\lambda$ jumps from a
positive value to zero through the free boundary, $\partial
\Omega_\lambda^\star$. The conclusion is that the optimal
regularity we can hope for $ u^\star_\lambda$ is Lipschitz
continuity. This is the contents of the next Theorem.
\begin{theorem}\label{Lipschitz Continuity}
    Let $\Omega_\lambda^\star$ be an optimal configuration to
    Problem (\ref{Penalized Min Problem}) and $u_\lambda^\star$
    its $\A$ potential. Then,
    $$
        \|\nabla u_\lambda^\star \|_{L^\infty(\R^n \setminus D)}
        \le C\lambda^{1/p},
    $$
for a constant $C$ that depends only on $\A$, $\Gamma$, $\varphi$
and $D$.
\end{theorem}
\begin{proof} We will provide two proofs of this important
theorem. The first one follows the glamorous approach suggested in
\cite{AC}. Unfortunately, for non-local problems like ours, the
efficiency of that method is restricted to the case $p\ge 2$ and a
new and more modern argument is required to establish Lipschitz
continuity for $\A$-potential associated to an optimal design with
when $1<p<2$. The second proof we will present works for all
$p>1$.

{\it $1^\textrm{st}$ Proof. The case $p \ge 2$.} We shall
initially obtain an competing estimate for Inequality
(\ref{uniform continuity Eq03}), with $u_\lambda^\delta$ replaced
by $u_\lambda^\star$. Enhancing the notation in the proof of
Theorem \ref{Exist for penalized problem}, $B = B_d(X_0)$ will be
a ball centered at a point in $\Omega_\lambda^\star$, $
    \dist(X_0, \partial D) >\!\!> d \ge \dist (X_0, \partial \Omega_\lambda^\star)
$ and $\mathfrak{h}$ the $\A$-harmonic function in $B$ that agrees
with $u_\lambda^\star$ on $\partial B$. For any direction $\nu \in
\mathbb{S}^{n-1}$, we define
$$
    r_\nu := \min \left \{ r \suchthat \frac{1}{4} \le r \le 1
    \textrm{ and } u_\lambda^\star(X_0 + d r \nu) = 0 \right \}
$$
if such a set is nonempty; otherwise, we put $r_\nu = 1$. For
almost every direction $\nu$ the map $r \mapsto
u_\lambda^\star(X_0 + dr \nu)$ is in $W^{1,p}[\frac{1}{4},1]$.
Thus, taking into account that $u_\lambda^\star(X_0 + d r_\nu \nu)
= 0$ whenever $r_\nu < 1$, we can compute,
\begin{equation}\label{Lip Cont Eq01}
    \begin{array}{lll}
        \mathfrak{h}(X_0 + d r_\nu \nu) & = & \displaystyle \int_{r_\nu}^1 \dfrac{d}{dr}
        (u_\lambda^\star - \mathfrak{h})(X_0 + d r \nu) dr \\
        & \le & d \cdot (1-r_\nu)^{1/p'} \times \left [ \displaystyle \int_{r_\nu}^1
        |\nabla ( \mathfrak{h} - u_\lambda^\star )(X_0 +  r\nu)|^p dr \right
        ]^{1/p},
    \end{array}
\end{equation}
where, as usual, $p'$ denotes the conjugate of $p$, i.e.,
$\frac{1}{p} + \frac{1}{p'} = 1$. Now, by the Harnack Inequality,
we know
\begin{equation}\label{Lip Cont Eq02}
    \inf\limits_{B_{\frac{7}{8}}} \mathfrak{h} \ge c_1
    \mathfrak{h}(X_0),
\end{equation}
for a constant $c_1>0$ that depends only on dimension and $\A$.
Here $B_{\frac{7}{8}}$ stands for $B_{\frac{7}{8}d}(X_0)$. Let us
consider the universal barrier, $\mathfrak{B}$, given by
\begin{equation}\label{Lip Cont Eq03}
    \left \{
        \begin{array}{rll}
            \div \left ( \A(X_0 + dX, D\mathfrak{B}(X) \right ) &=& 0 \textrm{ in } B_1(0) \setminus
            B_{\frac{7}{8}}(0) \\
            \mathfrak{B} &=& 0 \textrm{ on } \partial B_1(0) \\
            \mathfrak{B} &=& c_1 \textrm{ in }
            \overline{B_{\frac{7}{8}}(0)},
        \end{array}
    \right.
\end{equation}
where $c_1$ is the universal constant in (\ref{Lip Cont Eq02}). By
the Hopf's maximum principle, there exists a universal constant
$c_2 >0$, such that
\begin{equation}\label{Lip Cont Eq04}
    \mathfrak{B}(X) \ge c_2 \left (1 - |X| \right ).
\end{equation}
By the maximum principle and (\ref{Lip Cont Eq04}) we can write
\begin{equation}\label{Lip Cont Eq05}
    \mathfrak{h}(X_0 + dX) \ge \mathfrak{h}(X_0)\cdot
    \mathfrak{B}(X) \ge c_2 \mathfrak{h}(X_0)\cdot  (1 - |X|).
\end{equation}
Combining (\ref{Lip Cont Eq01}) and (\ref{Lip Cont Eq05}) we end
up with
\begin{equation}\label{Lip Cont Eq06}
    d^p \cdot
    \left [ \displaystyle \int_{r_\nu}^1|\nabla ( \mathfrak{h} - u_\lambda^\star )
    (X_0 +  r\nu)|^p dr \right ] \ge c_3 \mathfrak{h}^p(X_0) \cdot  (1 -
    r_\nu).
\end{equation}
Integrating (\ref{Lip Cont Eq06}) with respect to $\nu$ over
$\mathbb{S}^{n-1}$, taking into account the definition of $r_\nu$,
we find
\begin{equation}\label{Lip Cont Eq07}
    \left ( \dfrac{\mathfrak{h}(X_0)}{d} \right )^p
    \cdot \int_{B_d(X_0) \setminus B_{d/4}(X_0)} \chi_{\{ u_\lambda^\star =
    0\}} dX \le C_4 \int_{B_d(X_0)} \left |\nabla \left (\mathfrak{h} -  u_\lambda^\star
    \right )(X) \right |^p dX.
\end{equation}
If we replace, in all of our arguments so far, $B_{d/4}(X_0)$ by
$B_{d/4}(\overline{X})$, for any $\overline{X} \in \partial
B_{d/2}(X_0)$, we obtain
\begin{equation}\label{Lip Cont Eq08}
    \left ( \dfrac{\mathfrak{h}(X_0)}{d} \right )^p
    \cdot \int_{B_d(X_0) \setminus B_{d/4}(\overline{X})} \chi_{\{ u_\lambda^\star =
    0\}} dX \le \tilde{C}_4 \int_{B_d(X_0)} \left |\nabla \left (\mathfrak{h} -  u_\lambda^\star
    \right )(X) \right |^p dX, \quad \forall \overline{X} \in \partial B_{d/2}(X_0).
\end{equation}
Integrating (\ref{Lip Cont Eq08}) with respect to $\overline{X}$,
we prove the following important estimate:
\begin{equation}\label{Lip Cont Eq09}
    \left ( \dfrac{\mathfrak{h}(X_0)}{d} \right )^p
    \cdot \left | \left \{ X \in B_d(X_0) \suchthat u_\lambda^\star (X) = 0  \right \} \right |
    \le C_5 \int_{B_d(X_0)} \left |\nabla \left (\mathfrak{h} -  u_\lambda^\star
    \right )(X) \right |^p dX.
\end{equation}
Now we argue as follows: let $\rho := \dist(X_0, \partial \Omega)$
and for each $0 < \delta <\!\!< 1$, denote $\mathfrak{h}_\delta$
the $\A$-harmonic function in $B_{\rho + \delta}(X_0)$ that agrees
with $u_\lambda^\star$ on $\partial B_{\rho + \delta}(X_0)$.
Combining (\ref{uniform continuity Eq03}) and (\ref{Lip Cont
Eq09}) together with standard elliptic estimate, we deduce
\begin{equation}\label{Lip Cont Eq10}
    \begin{array}{lll}
        u_\lambda^\star(X_0) & = & \mathfrak{h}_\delta(X_0) + O(1)
        \\
        &\le & C\lambda^{1/p} (\rho + \delta) + O(1).
    \end{array}
\end{equation}
Letting $\delta  \searrow 0$ in (\ref{Lip Cont Eq10}) we finally
conclude
$$
    u_\lambda^\star(X_0) \le C \dist \left (X_0, \partial
    \Omega_\lambda^\star \right ),
$$
which clearly implies that $u_\lambda^\star$ is Lipschitz
continuous up to the free boundary $\partial \Omega_\lambda^\star$
and $\|\nabla u_\lambda^\star\|_\infty \lesssim \lambda^{1/p}$.
\\
{\it $2^\textrm{nd}$ Proof. The general case.} Let us assume, for
purpose of contradiction, that there exists a sequence of points
$X_k \in  \Omega_\lambda^\star$, with
$$
    X_k \to \partial \Omega_\lambda^\star, \quad \textrm{ and }
    \quad
    \dfrac{u_\lambda^\star(X_k)}{\dist (X_k,
    \Omega_\lambda^\star)}
    \nearrow + \infty.
$$
For convenience, we will call $N_k := u(X_k)$ and $d_k := \dist
(X_k, \Omega_\lambda^\star)$, thus our assumption is that
\begin{equation} \label{Lip Cont P Eq11}
    \dfrac{d_k}{N_k} = O(1).
\end{equation}
For each $k$, let $Y_k$ be a point on $\partial
\Omega_\lambda^\star$ that satisfies
$$
    |Y_k - X_k| = d_k.
$$
By replacing $X_k$ by another point $\tilde{X}_k$, if necessary,
because of the weak maximum principle we can assume that
\begin{equation} \label{Lip Cont P Eq12}
    N_k = \sup\limits_{B_{d_k}(Y_k)} u.
\end{equation}
On the other hand, by the Harnack inequality, there exists a
universal constant $\kappa > 0$, for which,
$\inf\limits_{B_{\frac{2}{3}d_k}} u \ge \kappa N_k$. Thus
\begin{equation} \label{Lip Cont P Eq13}
    \sup\limits_{B_{\frac{1}{3}d_k}(Y_k)} u \ge \kappa N_k.
\end{equation}
Now for each $\frac{1}{3} \le \gamma < 1$, let
$\mathfrak{h}_\gamma$ be the $\L$-harmonic function in $B_{\gamma
d_k}(Y_k)$ taking boundary data equals $u_\lambda^\star$. By
comparing, in terms of the optimal design problem (\ref{Penalized
Min Problem}), $u_\lambda^\star$  and the solution to the Obstacle
problem in Theorem \ref{obstable problem} with $\mathcal{M} =
\{u_\lambda^\star = 0 \} \setminus B_{\gamma d_k}(Y_k)$, we
deduce, as in the proof of Theorem \ref{Exist for penalized
problem}, that
\begin{equation} \label{Lip Cont P Eq14}
    \left ( \int_{B_{\gamma d_k}(Y_k)} \langle \A (X,
    Du_\lambda^\star), D u_\lambda^\star \rangle - \langle \A (X,
    D\mathfrak{h}_\gamma), D \mathfrak{h}_\gamma \rangle dX \right ) \le
    \lambda \left (\gamma d_k \right )^n.
\end{equation}
For each $k\ge 1$, we consider the functions
$\mathscr{U}^\gamma_k, ~ \mathscr{H}^\gamma_k \colon B_1 \to
(0,1)$ given by
\begin{equation} \label{Lip Cont P Eq15}
    \mathscr{U}^\gamma_k(Z) := \dfrac{1}{N_k} u_\lambda^\star(Y_k +
    \gamma d_k Z) \quad \textrm{and} \quad \mathscr{H}^\gamma_k(Z) := \dfrac{1}{N_k} \mathfrak{h}(Y_k +
    \gamma d_k Z).
\end{equation}
From (\ref{Lip Cont P Eq13}), we know that
\begin{equation} \label{Lip Cont P Eq15.5}
    \sup\limits_{B_{\frac{1}{3}\gamma}} \mathscr{U}^\gamma_k \ge
    \kappa.
\end{equation}
We also know that $\mathscr{H}^\gamma_k$ is the unique minimizer
of
\begin{equation} \label{Lip Cont P Eq15.7}
    \mathscr{D}(v) := \int_{B_1} \langle \A (Y_k + \gamma d_kX,
    Dv), D v \rangle dX,
\end{equation}
among functions $v \in W^{1,p}_0(B_1) + \mathscr{H}_k^\gamma$ and
it satisfies
\begin{equation} \label{Lip Cont P Eq15.8}
    \left \{
        \begin{array}{rlll}
            \div \left ( \A \left (Y_k + \gamma d_k,
            D\mathscr{H}^\gamma_k \right ) \right ) &=& 0 &\textrm{ in }
            B_1 \\
            \mathscr{H}^\gamma_k & = & \mathscr{U}^\gamma_k &\textrm{ on }
            \partial B_1.
        \end{array}
    \right.
\end{equation}
 A direct computation reveals that
\begin{equation} \label{Lip Cont P Eq16}
    \nabla \mathscr{U}^\gamma_k(Z) = \dfrac{\gamma d_k}{N_k}
    \nabla u_\lambda^\star(Y_k + \gamma d_k Z) \quad \textrm{and similarly} \quad \nabla \mathscr{H}^\gamma_k(Z) =
    \dfrac{\gamma d_k}{N_k}\mathfrak{h}(Y_k + \gamma d_k Z).
\end{equation}
Combining (\ref{Lip Cont P Eq16}) and the Change of Variables
Theorem we obtain that
\begin{equation} \label{Lip Cont P Eq17}
    \int_{B_{\gamma d_k}(Y_k)} \langle \A (X,
    Du_\lambda^\star), D u_\lambda^\star \rangle dX = \left \{ \dfrac{\gamma d_k}{N_k} \right \}^{-p} \cdot \left (\gamma d_k \right
    )^n \int_{B_1} \langle \A (Y_k + \gamma d_k,
    D\mathscr{U}^\gamma_k), D \mathscr{U}^\gamma_k \rangle dX,
\end{equation}
and the same holds when we replace $u_\lambda^\star$ by
$\mathfrak{h}_\gamma$ and $\mathscr{U}^\gamma_k$ by
$\mathscr{H}^\gamma_k$. In particular, taking into account
 (\ref{Lip Cont P Eq11}), (\ref{Lip Cont P Eq14}), we find that
\begin{equation} \label{Lip Cont P Eq18}
     \left ( \int_{B_1} \langle \A (Y_k + \gamma d_k,
    D\mathscr{U}^\gamma_k), D \mathscr{U}^\gamma_k \rangle -  \langle \A (Y_k + \gamma d_k,
    D\mathscr{H}^\gamma_k), D \mathscr{H}^\gamma_k \rangle dX \right ) =
    O(1)
\end{equation}
as $k \to \infty$. Furthermore,  reasoning as in the proof of
Theorem \ref{Exist for penalized problem}, we show that the family
of functions $\left \{ \mathscr{U}^\gamma_k \right \}_{k\ge 1}$ is
uniformly continuous in $B_1$. Now we argue as follows, fix a
$\gamma^* < 1$. From the uniform continuity, up to a subsequence,
\begin{equation} \label{Lip Cont P Eq19}
    \mathscr{U}^{\gamma^{*}}_k \to \mathscr{U} \quad \textrm{ and } \quad \mathscr{H}^{\gamma^{*}}_k \to
    \mathscr{H},
\end{equation}
uniformly in $\overline{B}_1$. Also, $Y_k \to Y_0$. From (\ref{Lip
Cont P Eq15.8}) and (\ref{Lip Cont P Eq15.7}), we obtain that
\begin{equation} \label{Lip Cont P Eq20}
    \left \{
        \begin{array}{rlll}
            \div \left ( \A \left (Y_0,
            D\mathscr{H} \right ) \right ) &=& 0 &\textrm{ in }
            B_1 \\
            \mathscr{H} & = & \mathscr{U} &\textrm{ on }
            \partial B_1
        \end{array}
    \right.
\end{equation}
and that $\mathscr{H}$ is the unique minimizer of
\begin{equation} \label{Lip Cont P Eq21}
    \mathscr{D}_0(v) := \int_{B_1} \langle \A (Y_0,
    Dv), D v \rangle dX,
\end{equation}
among functions $v \in W^{1,p}_0(B_1) + \mathscr{H}$. However,
from (\ref{Lip Cont P Eq18}), we obtain that
\begin{equation} \label{Lip Cont P Eq22}
    \mathscr{D}_0 (\mathscr{U}) =\mathscr{D}_0 ( \mathscr{H}).
\end{equation}
Therefore, $\mathscr{U} \equiv \mathscr{H}$. In particular,
$\mathscr{U}$ solves the elliptic PDE
$$
    \div \left ( \A \left (Y_0, D\mathscr{U} \right ) \right ) = 0 \textrm{ in }
    B_1.
$$
However, since $\mathscr{U}(0) = 0$, by the strong maximum
principle, $\mathscr{U} \equiv 0$, which ultimately contradicts
(\ref{Lip Cont P Eq15.5}) and the Theorem is finally proven.
\end{proof}
Our next step is to prove that $u_\lambda^\star$ growths linearly
away from $\partial \Omega_\lambda^\star$. Notice that this is the
most admissible growth rate allowed by the Lipschitz regularity
previously proven in Theorem \ref{Lipschitz Continuity}. Here is
the precise statement:
\begin{theorem}\label{nondegeneracy} There exists a constant $\underline{c} > 0$,
depending on dimension $\A$, $D$, $\Gamma$ and $\varphi$, such
that
$$
    \lambda^{-1/p} \underline{c} \cdot \dist \left (X_0, \partial
    \Omega_\lambda^\star \right ) \le u_\lambda^\star(X_0),
$$
for any $X_0 \in \Omega_\lambda^\star$.
\end{theorem}
\begin{proof} Let us fix $X_0 \in \Omega_\lambda^\star$ near the free boundary
and label $d := \dist (X_0, \partial
\Omega_\lambda^\star)$. From Theorem \ref{obstable problem}, there
exists a unique solution, $\phi$, to the following obstacle
problem
\begin{equation}\label{obs problem - nondegeneray}
    \textrm{Min } \left \{ \int_{D^C} \langle \A(X,Df),Df \rangle dX
    \suchthat f \in W^{1,p}(D^C)~ f = \varphi \textrm{ on } \partial D
    \textrm{ and } f \le 0 \textrm{ in } \{ u_\lambda^\star = 0 \} \cup
    B_{\frac{d}{2}}(X_0) \right \}.
\end{equation}
Recall, in Theorem \ref{Exist for penalized problem}, we proved
that $u_\lambda^\star$ is too a minimizer for problem
$(\mathfrak{P}_\lambda^\mathrm{weak})$, that is problem (\ref{Weak
Min Problem}) and clearly $\phi$ competes with $u_\lambda^\star$
in such a problem; therefore
\begin{equation} \label{nondeg Eq01}
    \int_{\partial D} \left ( \Gamma(X, \partial_\A \phi) -
    \Gamma(X, \partial_\A u_\lambda^\star) \right ) d\H \ge \lambda^{-1} c_n
    d^n,
\end{equation}
for a dimensional constant $c_n$. Since both $u_\lambda^\star$ and
$\phi$ are $\A$-harmonic in $D_\gamma$, where $\gamma$ is the
number in Proposition \ref{Min is positive around D}, for a
constant $C_1 = C_1(\Gamma)$, we can estimate
\begin{equation} \label{nondeg Eq02}
    \begin{array}{lll}
        \displaystyle \int_{\partial D} \left ( \Gamma(X, \partial_\A \phi) -
        \Gamma(X, \partial_\A u_\lambda^\star) \right ) d\mathcal{H}^{n-1} &\le& C_1
        \displaystyle  \int_{\partial D} \left ( \partial_\A \phi -
        \partial_\A u_\lambda^\star \right ) d\mathcal{H}^{n-1}\\
        &\le & \frac{C_1}{\inf\limits_{\partial D} \varphi}
        \displaystyle  \int \left ( \langle \A (X, D\phi), D\phi \rangle -
        \langle \A (X, Du_\lambda^\star), Du_\lambda^\star \rangle
        \right )
        dX.
    \end{array}
\end{equation}
Here we have used the measure representation provided by Lemma
\ref{Lf is a Radon Measure}. Now let $h$ satisfy
$$
    \L h = 0 \textrm{ in } B_{\frac{2}{3}d}(X_0) \setminus
    B_{\frac{d}{2}}(X_0), \quad h = 0 \textrm{ in }
    B_{\frac{d}{2}}(X_0), \quad \textrm{ and } \quad h = 1 \textrm{ on } \partial
    B_{\frac{2}{3}d}(X_0).
$$
By the Harnack inequality, there exists a constant $c_2 > 0$, such
that
\begin{equation}\label{employ H.I. for linear growth}
    u_\lambda^\star(X) \ge c_2 u_\lambda^\star(X_0) h(X) \textrm{ in } B_{\frac{2}{3}d}(X_0).
\end{equation}
Consider the auxiliary function
$$
    \mathfrak{g}(X) := \left \{
        \begin{array}{cll}
            \min \left \{ u_\lambda^\star(X), c_2u_\lambda^\star(X_0)h(X) \right
            \} &\textrm{ in }& B_{\frac{2}{3}d}(X_0) \\
            u_\lambda^\star(X) &\textrm{ in }&  D^C \setminus B_{\frac{2}{3}d}(X_0).
        \end{array}
    \right.
$$
Notice that $\mathfrak{g}$ competes with $\phi$ in the obstacle
problem, thus, combining (\ref{nondeg Eq01}), (\ref{nondeg Eq02})
and replacing $\phi$ by $\mathfrak{g}$, we obtain
\begin{equation} \label{nondeg Eq03}
    \lambda^{-1} c_3 \le \dfrac{1}{d^n} \displaystyle  \int_{\Pi}
    \left ( \langle \A (X, D\mathfrak{g}), D\mathfrak{g} \rangle -
    \langle \A (X, Du_\lambda^\star), Du_\lambda^\star \rangle
    \right ) dX.
\end{equation}
Where set set of integration in the above estimate can be taken to
be
$$
    \Pi := \left \{ X \in B_{\frac{2}{3}d}(X_0) \setminus
    B_{\frac{1}{2}d}(X_0) \suchthat c_2u_\lambda^\star(X_0)h(X)
    \le u_\lambda^\star(X) \right \}.
$$
However, in this set, we can estimate
\begin{equation} \label{nondeg Eq04}
    \begin{array}{lll}
        \langle \A (X, D\mathfrak{g}), D\mathfrak{g}(X) \rangle &\le & \Lambda
        |D \mathfrak{g}|^p \\
        & \le &  |Dh(X)|^p  \cdot \left [ c_2 u_\lambda^\star(X_0) \right ]^p
       \\
        & \le &  {C} \left  [ \dfrac{c_2 u_\lambda^\star(X_0)}{d} \right
        ]^p.
    \end{array}
\end{equation}
In the last inequality we have used the $C^{1,\alpha}$ estimate
for $h$. Finally, a combination of (\ref{nondeg Eq03}) and
(\ref{nondeg Eq04}) leads us to
$$
    \lambda^{-1/p}\underline{c} d \le u_\lambda^\star(X_0),
$$
for a constant  $\underline{c} =  \underline{c}(n,\A, D, \Gamma,
\varphi)$, and the Theorem is proven.
\end{proof}
Sometimes it is convenient to express nondegeneracy in any ball
centered at a point $X_0 \in \overline{\Omega_\lambda^\star}$.
This is the contents of the next Theorem.
\begin{theorem}\label{strong nondegeneracy}
Let $K$ be a compact set and $X_0 \in
\overline{\Omega_\lambda^\star} \cap K$. Then,
$$
    \sup\limits_{B_r(X)} u_\lambda^\star \ge c r,
$$
for some constant $c>0$ depending on dimension, $K$, $\A$, $D$,
$\Gamma$, $\varphi$ and $\lambda$.
\end{theorem}
\begin{proof}
The proof is basically the same of as the proof of Theorem
\ref{nondegeneracy}. The only difference is that $u_\lambda^\star$
is no longer $\A$-harmonic near a free boundary point $X_0$, thus
we replace the employment of Harnack inequality in (\ref{employ
H.I. for linear growth}) by:
$$
    \upsilon(X) := \sup\limits_{B_r} u_\lambda^\star \cdot h(X) \ge
    u_\lambda^\star \textrm{ on } \partial B_r,
$$
where $h$ is the $\A$-harmonic function in $B_r\setminus B_{r/2}$
taking boundary data 1 on $\partial B_r$ and $0$ in $B_{r/2}$. We
then define the auxiliary function $\mathfrak{g}(X) : = \min \left
\{u_\lambda^\star(X), \upsilon(X) \right \}$. The proof now
follows the same path as in  the proof of Theorem
\ref{nondegeneracy}.
\end{proof}
As usual, optimal regularity, Theorem \ref{Lipschitz Continuity}
and nondegeneracy, Theorem \ref{nondegeneracy} or Theorem
\ref{strong nondegeneracy}, as you like, allow a deeper
understanding on the geometric-measure properties of the free
boundary. In the next Theorem we will show that the free boundary
$\Omega_\lambda^\star$ has the appropriate weak geometry.
\begin{theorem} \label{weak geometry} There exists a constant $
0<\varsigma<1$, depending on dimension, $\A$, $D$, $\Gamma$,
$\varphi$, and $\lambda^{1/p}$, such that,
\begin{equation}\label{uniform positive density}
   \varsigma \omega_n r^n  \le \Leb \left ( B_r(Z) \cap
    \Omega_\lambda^\star \right ) \le (1 - \varsigma) \omega_n
    r^n,
\end{equation}
for any ball $B_r(Z)$ centered at a free boundary point $Z \in
\partial \Omega_\lambda^\star$. Furthermore,  the optimal
configuration $\Omega_\lambda^\star$ is a set of locally finite
perimeter and for positive constants $\underline{c}$,
$\overline{C}$, depending on $\A$, $D$, $\Gamma$, $\varphi$, and
$\lambda^{1/p}$, there holds
\begin{equation}\label{estimate on Hausdorff measure}
    \underline{c} r^{n-1} \le \mathcal{H}^{n-1} \left (\partial
    \Omega_\lambda^\star \cap B_r(Z) \right ) \le \overline{C}
    r^{n-1}
\end{equation}
for any ball $B_r(Z)$ centered at a free boundary point. In
particular, $\mathcal{H}^{n-1} \left ( \partial
\Omega_\lambda^\star \setminus \redbdry  \Omega_\lambda^\star
\right ) = 0$.
\end{theorem}
\begin{proof}
The estimate by below in (\ref{uniform positive density}), that
is, $\varsigma \omega_n r^n  \le \Leb \left ( B_r(Z) \cap
\Omega_\lambda^\star \right )$, is an immediate consequence of
Lipschitz regularity and strong nondegeneracy.
\par
Let us focus our effort to prove the uniform density of the zero
phase, $\R^n \setminus \Omega_\lambda^\star$. Let us assume, for
purpose of contradiction, the existence of a sequence of positive
real numbers $r_j$ with $r_j \searrow 0$ as $j \to \infty$ and
\begin{equation}\label{weak geo Eq01}
    \dfrac{\Leb \left (B_{r_j}(Z) \cap \{ u_\lambda^\star = 0 \} \right
    ) }{{r_j}^n} = O(1).
\end{equation}
We consider then the blow-up sequence $\mathfrak{q}_j  \colon B_1
\to \R$, defined as
\begin{equation}\label{weak geo Eq02}
    \mathfrak{q}_j (Y) := \dfrac{1}{r_j} u_\lambda^\star (Z +r_j
    Y).
\end{equation}
Let $\mathfrak{h}_j$ be the solution to
\begin{equation}\label{weak geo Eq03}
    \left \{
        \begin{array}{cll}
            \div \left ( \A(Z + r_jX, D\mathfrak{h}_j) \right ) & = & 0 \textrm{
            in } B_1 \\
            \mathfrak{h}_j & = & \mathfrak{q}_j  \textrm{
            on } \partial B_1.
        \end{array}
    \right.
\end{equation}
A renormalization of (\ref{uniform continuity Eq03}), when $p \ge
2$ or (\ref{uniform continuity Eq04}) when $1< p \le 2$, under the
assumption (\ref{weak geo Eq01}), reveals
\begin{equation}\label{weak geo Eq04}
    \int_{B_1} \left | \nabla \left ( \mathfrak{h}_j - \mathfrak{q}_j
    \right ) (Y) \right |^p dY = O(1).
\end{equation}
By Lipschitz regularity of $u_\lambda^\star$, and $C^{1,\alpha}$
elliptic estimate, up to a subsequence, we may assume
\begin{equation}\label{weak geo Eq05}
    \mathfrak{q}_j \stackrel{j\to \infty}{\longrightarrow}
    \mathfrak{q}_0 \quad \textrm{ and } \quad
    \mathfrak{h}_j \stackrel{j\to \infty}{\longrightarrow}
    \mathfrak{h}_0.
\end{equation}
uniformly in $B_{9/11}$. From (\ref{bound grad Eq03})
$\mathfrak{h}_0$ satisfies $\div (\A(Z, D\mathfrak{h}_0(Y)) =  0,$
and from (\ref{weak geo Eq04}) so does $\mathfrak{q}_0$, that is,
\begin{equation}\label{weak geo Eq06}
    \div \left ( \A(Z, D\mathfrak{q}_0) (Y)\right ) =  0 \textrm{ in } B_{1/2}.
\end{equation}
Since $\mathfrak{q}(0) = 0$, by the strong maximum principle, we
conclude $\mathfrak{q}(0) \equiv 0$ in $B_{1/2}$. However, this is
a contraction on the nondegeneracy property guaranteed by Theorem
\ref{strong nondegeneracy}.
\par
We now turn our attention to (\ref{estimate on Hausdorff
measure}). The estimate by above, that is $\mathcal{H}^{n-1} \left
(\partial \Omega_\lambda^\star \cap B_r(Z) \right ) \le
\overline{C} r^{n-1}$ is a consequence of Lipschitz regularity of
$u_\lambda^\star$. In order to prove the estimate by below in
(\ref{estimate on Hausdorff measure}), as before, let us assume,
for the sake of contradiction, that there exists a sequence $r_j
\searrow 0$ such that
\begin{equation}\label{weak geo Eq07}
    \dfrac{\mathcal{H}^{n-1} \left (\partial \Omega_\lambda^\star \cap
    B_{r_j}(Z)\right ) }{{r_j}^{n-1}} = O(1).
\end{equation}
With the notation as in (\ref{weak geo Eq02}), let us define the
sequence of nonnegative measures $\nu_j$, in $B_{2/3}$, as
\begin{equation}\label{weak geo Eq08}
    \nu_j:= \div \left ( \A(Z + r_jX, D\mathfrak{q}_j) \right ) dX.
\end{equation}
Via a compactness argument, we may assume, modulo passing to a
subsequence if necessary, that  $\nu_j \rightharpoonup \nu_0$ in
the sense of measures. However, condition (\ref{weak geo Eq07})
translates in terms of the measures $\nu_j$ as
\begin{equation}\label{weak geo Eq8.5}
    \nu_j \rightharpoonup 0.
\end{equation}
Moreover, by Lipschitz regularity, nondegeneracy and uniform
positive density of both phases, estimate (\ref{uniform positive
density}), it is not hard to verify that
\begin{equation}\label{weak geo Eq09}
    \nu_j \rightharpoonup \nu_0 := \div \left ( \A(Z,
    D\mathfrak{q}_0) \right ) dX.
\end{equation}
Indeed, from (\ref{uniform positive density}), $\Leb (\partial
\{\mathfrak{q}_0 > 0 \}) = 0$, thus in order to justify (\ref{weak
geo Eq09}), it is enough to attest such an identity holds true for
balls entirely contained in $\{\mathfrak{q}_0 > 0 \}$ and in
$\{\mathfrak{q}_0 = 0 \}$. If $B \subset \{\mathfrak{q}_0 > 0 \}$,
then by elliptic estimate, $\mathfrak{q}_j$ converges to
$\mathfrak{q}_0$ in a $C^{1,\alpha}$ fashion in $B$. Thus clearly
(\ref{weak geo Eq09}) is true. Now, if $B \subset \{\mathfrak{q}_0
= 0 \}$, then
$$
    \Big [ \div \left ( \A(Z, D\mathfrak{q}_0) \right ) dX \Big ] (B) =0,
$$
so we have to show that $\nu_j(B) \to 0$ as $j\to \infty$. This is
a consequence of nondegeneracy. In fact, let $\tilde{B} \subset\!
\subset B$. If there were a subsequence, $\mathfrak{q}_{j_k}$, for
each $\mathfrak{q}_{j_k} \not \equiv 0$ in $\tilde{B}$, then by
Theorem \ref{strong nondegeneracy}, there should exist points
$P_{k_j} \in \tilde{B}$, such that $\mathfrak{q}_{j_k}(P_{k_j})
\ge c > 0$. Then, passing to another subsequence, $P_{k_j} \to
\overline{P} \in \tilde{B}$, and since $\mathfrak{q}_{j_k}$
converges uniformly to $\mathfrak{q}_{0}$, we would reach the
conclusion that $\mathfrak{q}_{0}(P) > c$, which is not possible.
In conclusion, if $B_k$ is a nested sequence of balls, with $B_k
\nearrow B$, then, for some $j_k \in \mathbb{N}$,
$\mathfrak{q}_{j} \equiv 0$ in $B_k$, for any $j > j_k$.
Therefore, $\nu_j(B) \stackrel{j \to \infty}{\longrightarrow} 0$,
as desired.
\par
Having verified (\ref{weak geo Eq09}), the observation in
(\ref{weak geo Eq8.5}) tells us that
$$
    \div \left ( \A(Z, D\mathfrak{q}_0) \right ) = 0 \textrm{ in }
    B_{2/3},
$$
and as argued before, this leads us to a contradiction on the
nondegeneracy feature of $\mathfrak{q}_0$ assured in Theorem
\ref{strong nondegeneracy}.
\end{proof}
An immediate, yet quite important consequence of Theorem \ref{weak
geometry} is a substantial enhancement of Lemma \ref{Lf is a Radon
Measure} for the measure $\L u_\lambda^\star$.
\begin{theorem}\label{representation} There exists a Borel
function $Q_\lambda$, such that $ \L u_\lambda^\star = Q_\lambda
\lfloor \partial \Omega_\lambda^\star$. That is,
$$
    \int \div \left ( \A(X, Du_\lambda^\star) \right )
    \phi(X) dX = \int_{\partial \Omega_\lambda^\star} Q_\lambda(S)
    \phi(S) d \mathcal{H}^{n-1}(S),
$$
for any $\phi \in C^1_0(\R^n \setminus D)$. Moreover, $Q_\lambda$
bounded away from zero and infinity, that is for a positive
constant $C = C(\lambda, n,\A,D, \Gamma, \varphi)$, there holds
$$
    0< C^{-1} \le Q_\lambda \le C < \infty.
$$
\end{theorem}
As to provide some further insight, allow us to make some loose
comments regarding the representation Theorem
\ref{representation}. The Borel function $Q_\lambda$ should be
understood as a weak notion for the $\partial_\A u_\lambda^\star$
along the reduced free boundary $\redbdry \Omega_\lambda^\star$.
Indeed, in any $C^1$ peace of $\partial \Omega_\lambda^\star$,
there holds
\begin{equation}\label{Relation Qlambda and Normal Der. Eq01}
    Q_\lambda(S) = \langle \A \left (S, D u_\lambda^\star(S) \right ), \nu(S) \rangle,
\end{equation}
where $\nu$ is the unit inward normal vector to $\partial
\Omega_\lambda^\star$ at $S$. However, $\nu(S) = \frac{\nabla
u_\lambda^\star(S)}{\left | \nabla u_\lambda^\star(S) \right |}$,
thus, taking into account the scaling feature of $\A$, property
(c)(iv), from identity (\ref{Relation Qlambda and Normal Der.
Eq01}) we reach that
\begin{equation}\label{Relation Qlambda and Normal Der. Eq02}
    \left | \nabla u_\lambda^\star(S) \right | =
    \sqrt[p-1]{\dfrac{Q_\lambda(S)}{\langle \A \left (S, \nu(S) \right ), \nu(S)
    \rangle}}.
\end{equation}
In a more rigorous way, expression (\ref{Relation Qlambda and
Normal Der. Eq02}) can be proven to hold in terms of an asymptotic
approximation, that is, the following is true:
\begin{theorem}\label{FBC pointwise sense} Let $X_0 \in
\redbdry \Omega_\lambda^\star$. Then, for any $X \in
\Omega_\lambda^\star$ near $X_0$, we have
$$
     u_\lambda^\star(X) = \theta_\lambda(X_0) \left  \langle X - X_0, \nu(X_0)
    \right \rangle^{+}  + o(|X-X_0|),
$$
where $\theta_\lambda(X_0) =
\sqrt[p-1]{\frac{Q_\lambda(X_0)}{\langle \A \left (X_0, \nu(X_0)
\right ), \nu(X_0) \rangle}}$.
\end{theorem}
\begin{proof}
Indeed, consider a convergent blow-up sequence
\begin{equation}\label{FBC point Eq01}
    \mathfrak{q}_r (Y) := \dfrac{1}{r} u_\lambda^\star(X_0 + rY)
    \stackrel{r\searrow 0}{\longrightarrow} \mathfrak{q}_0.
\end{equation}
Easily, from standard geometric-measures arguments, combined with
nondegeneracy and the convergence in (\ref{FBC point Eq01}), we
see that
\begin{equation}\label{FBC point Eq02}
    \mathfrak{q}_0 \equiv 0 \textrm{ in } \left \{ X \in \R^n
    \suchthat \langle X, \nu(X_0) \rangle < 0 \right \}
    \quad \textrm{and} \quad
    \{\mathfrak{q}_0 > 0 \} = \left \{ X \in \R^n
    \suchthat \langle X, \nu(X_0) \rangle > 0 \right \}.
\end{equation}
Moreover
\begin{equation}\label{FBC point Eq03}
     \div \left ( \A \left ( X_0, D\mathfrak{q}_0 \right ),
     D\mathfrak{q}_0 \right ) = 0 \textrm{ in } \{\mathfrak{q}_0 > 0
     \}.
\end{equation}
Notice that $\partial \left \{\mathfrak{q}_0> 0 \right \}$ is the
hyperplane $\left \{ X \in \R^n \suchthat \langle X, \nu(X_0)
\rangle = 0 \right \}$: a smooth surface. One verifies from
Theorem \ref{representation} that
\begin{equation}\label{FBC point Eq04}
    \div \left ( \A \left ( X_0, D\mathfrak{q}_0 \right ),
    D\mathfrak{q}_0 \right ) = Q_\lambda(X_0) \lfloor
    \left \{ X \in \R^n \suchthat \langle X, \nu(X_0) \rangle = 0 \right
    \},
\end{equation}
hence, reasoning as before, we reach the following conclusion
\begin{equation}\label{FBC point Eq05}
      \nabla \mathfrak{q}_0 (Y) \cdot \nu(X_0) = \theta_\lambda(X_0), \quad \forall Y  \in \left \{ \langle X, \nu(X_0) \rangle = 0 \right \}.
\end{equation}
Recall $\mathfrak{q}_0$ is Lipschitz continuous in the entire
$\R^n$. Let $\mathfrak{q}^*_0$ be the odd reflation of
$\mathfrak{q}_0$ with respect to the hyperplane $\left \{ X \in
\R^n \suchthat \langle X, \nu(X_0) \rangle = 0 \right \}$. It is
easy to verify that $\|\nabla \mathfrak{q}^*_0\|_{L^\infty(\R^n)}
= \|\nabla \mathfrak{q}_0\|_{L^\infty(\R^n)} < C$ and that $\div
\left ( \A \left ( X_0, D\mathfrak{q}^{*}_0 \right ),
D\mathfrak{q}^{*}_0 \right ) = 0$  in the whole $\R^n$. From the
$C^{1,\alpha}$ regularity of $\mathfrak{q}^{*}_0$, we can employ
the beautiful and recent blow-up argument from \cite{KSZ} to
conclude that $\mathfrak{q}^{*}_0$ is an affine function. Thus in
view of (\ref{FBC point Eq05}), we obtain
$$
    \mathfrak{q}_0(X) = \theta_\lambda(X_0) \left  \langle X - X_0, \nu(X_0)
    \right \rangle^{+},
$$
and the Theorem is proven.
\end{proof}
We finish this section by proving the reduced free boundary,
$\redbdry \Omega_\lambda^\star$, admits a nice ``stratification".
More precisely, we have
\begin{theorem} There exists a collection of $C^1$ hypersurfaces
$\left \{ \mathfrak{S}_j \right \}_{j\ge 1}$, and compact subsets
${K}_j \subset \mathfrak{S}_j$, such that
$$
    \mathcal{H}^{n-1} \left ( \redbdry \Omega_\lambda^\star
    \setminus \bigcup\limits_{j\ge 1} K_j \right ) = 0.
$$
Furthermore, if $X \in K_j$, the unit outward theoretical normal
vector $-\nu(X)$ to $\redbdry \Omega_\lambda^\star$ is normal to
$\mathfrak{S}_j$.
\end{theorem}
\begin{proof}
Let $B = B_r(X_0)$ be a generically ball centered at a point of
the reduced free boundary. By the Lipschitz continuity of
$u_\lambda^\star$ and the ellipticity of $\A$, we know there
exists a constant $L$, such that
\begin{equation}\label{C1 stratif. Eq01}
    \sup\limits_{B} \A(X, D u_\lambda^\star) \le \dfrac{L}{6}.
\end{equation}
Let $\mathcal{I}$ be your favorite nonnegative radially symmetric
smooth function whose support is $B_1$. Normalize it so that $0\le
\mathcal{I} \le 1$; $\|\mathcal{I}\|_{L^1(B_1)} = 1$. Let
$\mathcal{I}_\epsilon$ be the family of mollification induced by
$\mathcal{I}$, that is, $\mathcal{I}_\epsilon(X) = \epsilon^{-n}
\mathcal{I}_\epsilon(\epsilon^{-1}X).$ Also, select your favorite
nonnegative function $\eta \in C^{\infty}_0(B)$, satisfying $\sup
\eta = L^{-1}$. For sake of notation convenience, let us call $V
(X) :=  \A(X, D u_\lambda^\star)$. If $\nu$ denotes the Radon
measure $D \chi_{\Omega_\lambda^\star}$, we have, for $\epsilon
<\!\!< 1$,
\begin{equation}\label{C1 stratif. Eq02}
    \begin{array}{lll}
        \nu (B) & := & \sup \left \{ \displaystyle \int_{\Omega_\lambda^\star}
        \div ~\psi dX \suchthat \psi \in C^1_0(B; \mathbb{R}^n ),
        ~ \| \psi \| \le 1 \right \} \\
        & \ge & \displaystyle \int_{\Omega_\lambda^\star}
        \div  \left ( (\eta V) * \mathcal{I}_\epsilon  \right
        ) dX \\
        & = & \displaystyle \int_{\Omega_\lambda^\star}
        \mathcal{I}_\epsilon * \div  \left ( \eta V \right
        ) dX \\
        & = & \displaystyle \int_{\Omega_\lambda^\star}
        \div  \left ( \eta V \right
        ) dX + O(1) \\
        & = & \displaystyle \int_{\Omega_\lambda^\star} V \cdot \nabla \eta dX +
        \displaystyle \int_{\redbdry \Omega_\lambda^\star} Q_\lambda(S)
        \eta(S) d \mathcal{H}^{n-1}(S) + O(1).
    \end{array}
\end{equation}
Letting $\epsilon \to 0$ in (\ref{C1 stratif. Eq02}) and
afterwards letting $\eta \to L^{-1}$, we conclude there exists a
constant $c(\lambda, \A, n, \Gamma, \varphi)$, such that
\begin{equation}\label{C1 stratif. Eq02}
    \nu(B) \ge c \mathcal{H}^{n-1} \left ( B \cap \redbdry
    \Omega_\lambda^\star \right ).
\end{equation}
In particular $ \mathcal{H}^{n-1} \lfloor \redbdry
\Omega_\lambda^\star$ is absolutely continuous with respect to $D
\chi_{\Omega_\lambda^\star}$. Now, arguing as in \cite{DeGiorgi}
(see also \cite{Giusti} page 54 or \cite{EG} page 205) we prove
the Theorem.
\end{proof}

\section{Existence of an optimal configuration for problem (\ref{Minimization
Problem}) in any dimension} \label{SECTION - Existence theory in
any dimension}

In section \ref{SECTION - Existence up to dimension p}, upon a
restriction on the dimension, we have shown problem
(\ref{Minimization Problem}) has a minimal configuration. The
strategy there was to let the penalizing parameter $\lambda$ go to
infinity and use appropriate estimates that becomes available
under the constraint $n < p$, due to the Sobolev Imbedding
Theorem.
\par
The goal of this section is to explore the geometric-measure
properties of the free boundary $\partial \Omega_\lambda^\star$,
established in the previous section, to settle to existence of an
optimal design for problem (\ref{Minimization Problem}) in all
dimensions. However, as the readers should expect, the analysis
here is rather more delicate as we will not be able to pass the
limit on the penalty parameter $\lambda$. Instead, we will show
that if we adjust the penalty term $\varrho_\lambda$ properly, any
optimal configuration, $\Omega^{\star} = \Omega^{\star}_\lambda$,
for problem (\ref{Penalized Min Problem}) will obey
$$
   \Leb \left ( \Omega^{\star} \setminus D \right ) \le \iota.
$$
Therefore, $\Omega^{\star}$ itself will be an optimal design for
our primary optimization problem (\ref{Minimization Problem}) and
all the regularity features proven to hold for a solution to
problem (\ref{Penalized Min Problem}) will automatically extend to
a solution to problem (\ref{Minimization Problem}).
\par
Before continuing, let us explain our strategy in a bit more
technical terms. We will perform a small perturbation on an
optimal configuration $\Omega_\lambda^\star$, around a point on
the reduced free boundary: the portion of $\partial
\Omega_\lambda^\star$ where we can replace classical differential
geometry arguments by geometric-measures ones. We will not compute
the Borel function $Q_\lambda$ of Theorem \ref{representation}, as
it is an extraordinary hard task: the free boundary condition for
problem (\ref{Penalized Min Problem}) is expected to be highly
nonlocal. Instead, we will show that assuming
$\Leb(\Omega_\lambda^\star \setminus D) > \iota$ enforces a
universal bound to the penalty parameter $\lambda$.
\par
With the strategy well understood, let us establish the first
supporting result towards the main goal of this section.
\begin{lemma}\label{estimate by above on Q lambda}
There exists a constant $M >0$, depending on dimension, $D$,
$\varphi$, $\Gamma$ and $\A$, but independent of $\lambda$, such
that
$$
    \inf\limits_{\redbdry
    \Omega_\lambda^\star} Q_\lambda < M,
$$
where $Q_\lambda$ is the Borel function in Theorem
\ref{representation}.
\end{lemma}
\begin{proof}
Indeed, in the lights of Lemma \ref{control on gradient}, there
exists a constant $C$, independent of $\lambda$, such that
$\|u_\lambda^\star\|_{W^{1,p}} \le C$. Thus, from the Trace
Theorem for Sobolev functions, we can write
$$
   \|u_\lambda^\star\|_{W^{1,p}} \cdot
   \left [\Leb \left( \{ u_\lambda^\star > 0 \} \right ) \right
    ]^{\frac{1}{p'}} \ge
     \int_{\partial D} \varphi(Z) d\mathcal{H}^{n-1}(Z).
$$
The above estimate combined with the Isoperimetric Inequality
assures the existence of a constant $\underline{c}_1 > 0$,
independent of $\lambda$, for which the following estimate holds
\begin{equation} \label{estimate by above on Q lambda Eq01}
    \mathcal{H}^{n-1} \left( \redbdry \Omega_\lambda^\star \right ) \ge
    \underline{c}_1.
\end{equation}
From (\ref{estimate by above on Q lambda Eq01}) and the
representation in Theorem \ref{representation}, we have
\begin{equation} \label{estimate by above on Q lambda Eq02}
    \begin{array}{lll}
        \displaystyle \int_{\partial D} \partial_\A u_\lambda^\star(X) d\H &=&
        \displaystyle \int_{\redbdry \Omega_\lambda^\star}  Q_\lambda(X)
        d\H\\
        &\ge& \underline{c}_1 \inf\limits_{\redbdry
        \Omega_\lambda^\star} Q_\lambda.
    \end{array}
\end{equation}
Now, in view of estimate (\ref{bound grad Eq04}), for each $Y \in
\partial D$ fixed, we establish the following estimate
$$
    \int_{\partial D} \Gamma \left ( Y,
    \underline{c}_2 \cdot \left [
    \inf\limits_{\redbdry \Omega_\lambda^\star} Q_\lambda \right ]
    \right ) d\H \le \int_{\partial D} \Gamma \left ( Y,
    \intav{\partial D}{\partial_\A u_\lambda^\star}
    \right ) d\H \le C_2.
$$
Integrating the above estimate with respect to $Y$ over $\partial
D$ and arguing as before,  we conclude the proof of the Lemma.
\end{proof}

We now pass to describe the mathematical setup for the suitable
perturbation technique we shall employ on $\Omega_\lambda^\star$
near a point on the reduced free boundary. Initially, select and
fix, throughout this section, a free boundary point $Z_0 \in
\redbdry \Omega_\lambda^\star$, such that
\begin{equation}\label{Existence ALL Dimensions Eq01}
    Q_\lambda(Z_0) \le 5 \inf\limits_{\redbdry \Omega_\lambda^\star}
    Q_\lambda \le M_1,
\end{equation}
where $M_1$ depends only on dimension, $\A$, $D$, $\Gamma$ and
$\varphi$, but it is independent of $\lambda$. The existence of
such a point is guaranteed by Lemma \ref{estimate by above on Q
lambda}.
\par
Let $\psi \colon \mathbb{R} \to \mathbb{R}$ be your favorite
nonnegative smooth function whose support equals $[0,1]$.
Normalize it so that
$$
    \int \psi(\tau)d \tau = 1.
$$
For a fixed positive, but small, real number $\alpha$, we define
the inward perturbation map around $Z_0$ as
\begin{equation}\label{Perturbation}
    \Phi_r(X) := \left \{
        \begin{array}{cr}
            X - \alpha r \psi \left ( \dfrac{|X-Z_0|}{r} \right ) \nu(Z_0) & X \in B_r(Z_0)
            \\
            X & X \not \in B_r(Z_0).
        \end{array}
        \right.
\end{equation}
Here, $\nu(Z_0)$ denotes the theoretical measure outward normal
vector at $Z_0$. The idea now is to compare $\Omega_\lambda^\star$
with its inward perturbed configuration given by:
\begin{equation}\label{def of inward pert configuration}
    \Omega_r :=  \Phi_r\left ( \Omega_\lambda^\star \right ).
\end{equation}
For that, let us call $u_r$ the $\A$-potential associated to
$\Omega_r$, that is, $u_r$ is the solution to
\begin{equation}\label{A potential of perturbated configuration}
\left \{
\begin{array}{rll}
\L u_r &=& 0 \textrm{ in } \Omega_r \setminus D \\
u_r & = & \varphi \textrm{ on } \partial D \\
u_r &=& 0 \textrm{ on } \partial \Omega_r
\end{array}
\right.
\end{equation}
Although it is possible to compare $u_r$ and $u$ directly, it
turns out to the more convenient to use the auxiliary function,
$v_r$, implicitly  by
\begin{equation}\label{vr}
    v_r \left( \Phi_r(X) \right ) = u_\lambda^\star(X).
\end{equation}
Notice that $\left ( \{ v_r > 0 \}, v_r \right )$ is not suitable
for our minimization problem (\ref{Penalized Min Problem}). Also
it not efficient to compare it with $u_\lambda^\star$ in terms of
the minimization problem (\ref{Weak Min Problem}), since
$\partial_\A u_\lambda^\star \equiv \partial_\A v_r.$ Our strategy
is to compare $v_r$ with $u_\lambda^\star$ and with $u_r$
separately and then combine these information using $v_r$ as a
bridge from $u_r$ and $u_\lambda^\star$.
\par
The next two Lemmas are from \cite{OT}, Section 4, though in that
paper the computations are carried out only for the $p$-Laplacian
operator. Thus we decide to include in this present work
``economic versions" of their proofs as a courtesy to the readers.
\begin{lemma}\label{estimate on Leb meausre on the perturbation}
With the notation previously set, we have
$$
     \Leb \left ( \{ u > 0 \} \right ) - \Leb \left (  \{ v_r > 0 \}  \right )
     = M_2 \alpha r^{n} + o(r^{n}),
$$
for a universal constant $M_2 > 0$.
\end{lemma}
\begin{proof}
For sake of notation convenience, we will write $u$ for
$u_\lambda^\star$. For each $r>0$ small, we consider the
$r$-normalization of $u$ around $Z_0$, $u_r \colon B_1 \to
\mathbb{R}$, defined as

$$
    u_r(Y) := \dfrac{1}{r} u \big (Z_0 + rY \big ).
$$
Since $Z_0 \in \redbdry \Omega_\lambda^\star$,
\begin{equation}\label{Had0}
    B_1 \cap \left \{ u_r > 0 \right \} \stackrel{r\to 0}{\longrightarrow}
    \left \{ Y \in B_1 ~ \big | ~ \langle Y, \nu(Z_0) \rangle < 0 \right \},
\end{equation}
in the sense that the characteristic functions of the above sets
in the LHS converge to the characteristic function of the set in
the RHS in the $L^1_\textrm{loc}(\mathbb{R}^n)$ topology.  One
easily sees, by the Change of Variables Theorem, that
\begin{equation}\label{Had1}
    \begin{array}{lcl}
        &\dfrac{\Leb \left( B_r(Z_0) \cap \{ v_r > 0 \} \right )}{r^n}&  =
        \dfrac{1}{r^n} \displaystyle \int\limits_{ B_r(Z_0) \cap \{ v_r > 0 \}} dX \vspace{.2cm}\\
        & = & \displaystyle \int\limits_{ B_1 \cap \{ v_r(Z_0 + rY) > 0 \}} dY \vspace{.2cm}\\
        & = & \displaystyle \int\limits_{ B_1 \cap \{ u_r > 0 \}} \det \left ( D\Phi_r (Z_0 + r Y)
        \right ) dY \vspace{.2cm}\\
        & \stackrel{r\to 0}{\longrightarrow} &  \hspace{-.5cm} \displaystyle
        \int\limits_{ B_1 \cap \{ \langle Y, \nu(Z_0) \rangle < 0 \} }
        \hspace{-.5cm} 1  - \alpha \psi'(|Y|) \left \langle \dfrac{Y}{|Y|}, \nu(Z_0)
        \right \rangle dY,
    \end{array}
\end{equation}
It is important to highlight that for any unit vector $\nu \in
\mathbb{S}^{n-1}$,
\begin{equation}\label{Had2}
\displaystyle \int\limits_{ B_1 \cap \{ \langle Y, \nu \rangle < 0
\} } \hspace{-.5cm} \psi'(|Y|) \left \langle \dfrac{Y}{|Y|}, \nu
\right \rangle dY \equiv M_2,
\end{equation}
where $M_2$ is a constant that depends only on your choice for
$\psi$. Similarly, one finds that
\begin{equation}\label{Had21}
    \dfrac{\Leb \left( B_r(Z_0) \cap \{ u_\lambda^\star > 0 \} \right
    )}{r^n} \stackrel{r\to 0}{\longrightarrow}
    \int\limits_{ B_1 \cap \{ \langle Y, \nu(Z_0) \rangle < 0 \} }
    dY.
\end{equation}
Combining (\ref{Had0}), (\ref{Had1}), (\ref{Had2}) and
(\ref{Had21}), we conclude the Lemma.
\end{proof}
Our next Lemma measures the differential on the $\A$-Dirichlet
integral passing from $u_\lambda^\star$ to  $v_r$.
\begin{lemma}\label{estimate on the  p-Dirichlet on the
perturbation} There exists a constant $M_3 > 0$ depends on
dimension, $D$, $\Gamma$, $\varphi$ and $\psi$, but it is
independent of $\lambda$ such that
$$
    \dfrac{1}{r^n} \displaystyle \int \left \{ \langle \A(X, D v_r), Dv_r \rangle  -
    \langle \A(X, D  u_\lambda^\star), D u_\lambda^\star \rangle
    \right \} dX  \le \alpha M_3 + o(\alpha) + O(1).
$$
\end{lemma}
\begin{proof}
Again, for sake of notation convenience, we will write $u$ for
$u_\lambda^\star$. Yet for notation convenience, let us write, for
any vector field $\overrightarrow{V}$,
$\Theta(\overrightarrow{V})(X) := \langle \A(X,
\overrightarrow{V}), \overrightarrow{V} \rangle.$ Applying the
Change of Variables Theorem twice and taking into account that
$P_r$ maps $B_r(X_i)$ diffeomorphically onto itself, we can write
\begin{equation}\label{Had7}
    \begin{array}{lll}
        \dfrac{1}{r^n} \displaystyle\int\limits_{B_r(Z_0)} \Theta (Dv_r)(X) dX
        & = & \dfrac{1}{r^n} \displaystyle \int\limits_{B_r(Z_0)}
        \Theta \left ( D\Phi_r(\Phi_r^{-1}(X))^{-1}
        \cdot \nabla u (\Phi_r^{-1}(X)) \right ) dX \\
        & = & \dfrac{1}{r^n} \displaystyle \int\limits_{B_r(Z_0)} \Theta \left ( D\Phi_r(Y)^{-1}
        \cdot \nabla u (Y) \right )
        \times \left | \det \big ( D\Phi_r(Y) \big )  \right | dY  \\
        &=& \displaystyle \int\limits_{B_1 \cap \{ u_r > 0 \} }
        \Theta \left ( D\Phi_r(Z_0 + rZ)^{-1}
        \cdot \nabla u_r (Z) \right ) \times  \left |\det \big ( D\Phi_r(Z_0 + rZ) \big )
        \right |
        dZ.
    \end{array}
\end{equation}
By an explicit computation it is easy to verify that
\begin{equation}\label{Had8}
D\Phi_r(Z_0 + rZ)^{-1} \cdot \nabla u^i_r (Z)  = \nabla u_r (Z) +
 \alpha \dfrac{\psi'(|Z|)}{|Z|} \langle Z, \nabla u_r (Z) \rangle
\nu(Z_0) + o(\alpha).
\end{equation}
Furthermore, we can compute explicitly that
\begin{equation}\label{Had9}
|\det \big ( D\Phi_r(Z_0 + rZ) \big )|  = 1 -\alpha
\dfrac{\psi'(|Z|)}{|Z|} \langle Z,\nu(Z_0) \rangle.
\end{equation}
A straight combination of  (\ref{Had7}), (\ref{Had8}) and
(\ref{Had9}), revels that
\begin{equation}\label{Had10}
        \dfrac{1}{r^n} \hspace{-.2cm} \displaystyle \int\limits_{B_r(Z_0)}
        \hspace{-.2cm} \Theta \left ( D v_r \right )(X) -
        \Theta \left (  D u \right )(X) dX  = -\alpha \hspace{-.2cm}
         \displaystyle \int\limits_{B_1 \cap \{ u_r > 0 \} } \Theta\left (D
         u_r(Z) \right )
         \dfrac{\psi'(|Z|)}{|Z|}
        \langle Z,\nu(Z_0)\rangle dZ + o(\alpha).
\end{equation}
It is simple to verify, from the Divergence Theorem,  that
\begin{equation}\label{Had13}
    \displaystyle \int\limits_{B_1 \cap \{ u_r > 0 \} } \dfrac{\psi'(|Z|)}{|Z|} \langle Z,\nu(Z_0)
    \rangle dZ  \longrightarrow - \int\limits_{B_1 \cap \{ \langle Z, \nu(Z_0) \rangle =  0 \} }
    \psi(|Z|) d\mathcal{H}^{n-1}(Z)
    = I > 0,
\end{equation}
with the appropriate integral orientation. Furthermore, by the
Lipschitz regularity of $u$ and standard geometric-measure
arguments we verify that
\begin{equation}\label{Had11}
    \langle \A(Z_0 + rY ,\nabla u_r) , \nabla u_r \rangle  \to  Q_\lambda(Z_0)  \nu(Z_0) \mathbf{\chi}_{ B_1 \cap \{ \langle Y, \nu(X_i) \rangle < 0 \} },
\end{equation}
in $L^p(B_1)$. Thus, letting $r \to 0$ in (\ref{Had10}), and
taking into account (\ref{Had13}) and estimate (\ref{Existence ALL
Dimensions Eq01}), we conclude the proof of the Lemma.
\end{proof}
We are ready to prove the existence of an optimal design for
problem (\ref{Minimization Problem}) in all dimensions.
\begin{theorem}\label{existence of min for ALL dimensions}
    There exists a positive number $\lambda_0$, such that if
    $\Omega_\lambda^\star$ is an optimal configuration for problem
    (\ref{Penalized Min Problem}) and $\Leb \left (\Omega_\lambda^\star
    \setminus D \right ) > \iota$, then necessarily, $\lambda <
    \lambda_0$. In particular, there exists an optimal
    configuration for problem (\ref{Minimization Problem}) and it
    enjoys all the weak geometric features derived in Section
    \ref{SECTION - Weak Geometric Properties of the free
    boundary}.
\end{theorem}
\begin{proof}
    Throughout the proof we fix an optimal configuration $\Omega_\lambda^\star$ and assume
    \begin{equation}\label{Existence ALL Dimensions Eq00}
        \Leb \left ( \Omega_\lambda^\star \setminus D \right ) >
        \iota.
    \end{equation}
    Initially we recall the variational characterization of the
    $\A$-potential $u_r$, namely
    \begin{equation}\label{Final Thm 01}
        \int \langle \A(X,Du_r), Du_r \rangle dX = \min \left \{
        \int \langle \A(X,Dv), Dv \rangle dX \suchthat v = \varphi
        \textrm{ on } \partial D \textrm{ and } v = 0 \textrm{ on } \partial
        \Omega_r \right \}.
    \end{equation}
    Now we compare $\Omega_\lambda^\star$ with $\Omega_r$ in terms
    of the minimization problem (\ref{Penalized Min Problem}). From the minimality feature
    of the configuration $\Omega_\lambda^\star$, if $r$ is small enough as to
    $\Leb \left ( \Omega_r \setminus D \right ) > \iota$, we
    have
    \begin{equation}\label{Final Thm 02}
    \lambda \left \{ \Leb \left ( \Omega_\lambda^\star \setminus D \right ) -
    \Leb \left ( \Omega_r \setminus D \right ) \right \} \le
    \int_{\partial D} \Gamma (X,
    \partial_\A u_r) - \Gamma (X,
    \partial_\A u_\lambda^\star) d\H.
    \end{equation}
    As argued before, we have the following estimate
    \begin{equation}\label{Final Thm 03}
        \begin{array}{lll}
            \displaystyle \int_{\partial D} \Gamma (X,
            \partial_\A u_r) - \Gamma (X,
            \partial_\A u_\lambda^\star) d\H & \le & C(\partial D, \Gamma)
            \displaystyle \int_{\partial D}
            \left \{ \partial_\A u_r - \partial_\A u_\lambda^\star \right \} d\H \\
            &\le & C(\partial D, \Gamma, \inf \varphi) \displaystyle
            \int \langle \A(X,Du_r),Du_r \rangle  \\
            & & -
            \langle \A(X,Du_\lambda^\star),Du_\lambda^\star
            \rangle dX.
        \end{array}
    \end{equation}
    Now combining Lemmas \ref{estimate on Leb meausre on the perturbation}
    and \ref{estimate on the  p-Dirichlet on the perturbation}
    with (\ref{Final Thm 01}), (\ref{Final Thm 02}) and (\ref{Final Thm 03}), we obtain
    \begin{equation}\label{Final Thm 04}
            \lambda \left \{M_2 \alpha r^n + o(r^n) \right \} \le C(\partial D, \Gamma, \inf \varphi)
            r^n \times \left [\alpha M_3 + o(\alpha) +O(1)
            \right ].
    \end{equation}
    If we divide expression (\ref{Final Thm 04}) by $r^n$, let $r
    \to 0$ and afterwards divide the result by $\alpha$ and let $\alpha \searrow 0$, we
    finally conclude the proof of the Theorem.
\end{proof}

\bibliographystyle{amsplain, amsalpha}

\end{document}